%% file: Time_Fractional.tex
\documentclass[a4paper]{article}

\title{Non-uniform $\alpha$-Robust Alikhanov Mixed FEM with Optimal Convergence for the Time-Fractional Allen--Cahn Equation}

\author{
Abhinav Jha\footnote{Department of Mathematics, Indian Institute of Technology Gandhinagar, Palaj, Gandhinagar, 382055, Gujarat, India \texttt{abhinav.jha@iitgn.ac.in}},
Samir Karaa\footnote{Department of Mathematics, Sultan Qaboos University, Al-Khod 123, Muscat, Oman \texttt{skaraa@squ.edu.om}},
Aditi Tomar\footnote{TU Wien, Institut für Analysis und Scientific Computing, Wiedner Hauptstr. 8--10, 1040 Wien, Austria \& Department of Mathematics, Indian Institute of Technology Gandhinagar, Palaj, Gandhinagar, 382055, Gujarat, India \texttt{aditi.tomar@asc.tuwien.ac.at}, \texttt{aditi.tomar@iitgn.ac.in}}
}

\input{packages}

\usepackage[a4paper, total={6.5in, 8.5in}]{geometry}
\usepackage{mathrsfs}

\begin{document}

\maketitle 

\begin{abstract}
We investigate a mixed finite element method for the spatial discretization of a time-fractional Allen–Cahn equation defined on a convex polyhedral domain, combined with a nonuniform Alikhanov scheme for the temporal approximation. Under suitable regularity assumptions on the initial data that are weaker than those typically imposed in the literature, we establish regularity results for the solution and its flux. We then derive optimal $L^2$-error estimates, up to a logarithmic factor, for both the solution and the flux. The estimates are robust with respect to the fractional order \(\alpha\), in the sense that the associated constants remain bounded as \(\alpha\rightarrow 1^{-}\). Numerical experiments are presented to confirm the theoretical findings. 
\end{abstract}

\medskip
\textbf{Keywords}: Caputo fractional derivative; Alikhanov method; graded mesh; discrete fractional Gr\"{o}nwall inequality; regularity results; error analysis.
 \section{Introduction}\label{section1}
Classical phase-field models are diffuse interface models that have found numerous applications
in diverse research areas, e.g., hydrodynamics \cite{MR1609626, MR1984386}, material sciences \cite{shah2016efficient}, and multi-phase
flow \cite{MR4889514} and several others. The classical Allen-Cahn equation was initially introduced by Allen and Cahn \cite{allen1979microscopic} to model the dynamics of antiphase boundaries in crystalline solids. Since then, it has evolved into a versatile mathematical framework with a wide range of applications, including the behavior of vesicle membranes and the processes of solid nucleation and phase separation in two incompressible fluids.

Let $J=(0,\;T]$ denote a time interval, where $0<T<\infty$, and let $\Omega\subset\mathbb{R}^d$ for $(d\geq 1)$ be a convex polygonal or polyhedral bounded domain. In this paper, we focus on phase field models that involve time nonlocality. More specifically, we consider the following time-fractional Allen--Cahn model \cite{liu2018time, du2020time}:
\begin{align}\label{main}
\begin{cases}
\partial^{\alpha}_{t} u - \kappa^2 \Delta u  \;=\; -F^\prime(u):=f(u)  & \quad \text{in} \quad \Omega\times J,\\[5pt]
u ~= ~0 &\quad \text{on}\quad \partial\Omega\times J, \\[5pt]
u(\boldsymbol{x},\;0) ~= ~u_0(\boldsymbol{x}) &\quad \text{in} \quad \Omega.
\end{cases}
\end{align}
Here, $u$ denotes the phase of material, the Caputo time-fractional derivative $\partial_t^\alpha$ of order $\alpha\; (0<\alpha < 1)$ is defined as 
$\partial_t^{\alpha} u(t) = \mathcal{I}^{1-\alpha} u_t(t)$, where 
 $\mathcal{I}^{1-\alpha}$ represents the Riemann--Liouville fractional integral of order $1-\alpha$, defined for sufficiently regular $\varphi$ by
\[
\mathcal{I}^{1-\alpha}\varphi(t)
=\int_0^t \omega_{1-\alpha}(t-s)\varphi(s)\,ds,\qquad
\omega_{\beta}(t):=\frac{t^{\beta-1}}{\Gamma{(\beta)}},
\] and $\Gamma(\cdot)$ is the Gamma function.
Moreover,  the constant $0 < \kappa <1$ is the interaction length that describes the thickness of the transition boundary between materials. The nonlinear term $F(u):=\frac{1}{4}(1-u^2)^2$ is the interficial (or potential) energy.

Several numerical approaches have been proposed for the temporal discretization of time-fractional derivatives, including the widely used L1 method and its variants \cite{MR3639581, huang2020optimal, TOMAR2024137, MR3790081, wang2022local}, as well as higher-order schemes such as the Alikhanov method \cite{liao2020second, MR4402734, MR3273144, MR3904430, MR4270344}. In the present work, we adopt the Alikhanov scheme for the temporal discretization of problem (\ref{main}), owing to its higher-order accuracy and favorable stability properties. Since solutions of time-fractional problems typically exhibit weak singular behavior near the initial time, the use of graded temporal meshes with refined time steps near $t=0$ has proven particularly effective in capturing this initial-layer behavior. 

The analytical and numerical properties of problem (\ref{main}) have been investigated in several earlier studies. In particular, the well-posedness of the problem and its limited smoothing properties were established in \cite{du2020time}, where the authors developed a family of unconditionally stable time-stepping schemes based on backward Euler convolution quadrature (BE-CQ) for the fractional derivative. By combining a discrete fractional Gr\"{o}nwall inequality with maximal 
$p$-regularity arguments, they derived convergence results without requiring additional regularity assumptions on the exact solution. In the study by Huang and Stynes \cite{huang2020optimal}, the authors utilized the L1 method to discretize the temporal fractional derivative on a graded mesh, combined with a standard finite element discretization of the spatial diffusion term where $u_0 \in H_0^1(\Omega) \cap H^4(\Omega)$. They also applied Newton's linearization to address the nonlinear driving term and demonstrated optimal-rate convergence of the solution in the \( H^1 \) norm. Furthermore, in another work \cite{MR4402734}, they analyzed the Alikhanov method for discretizing the temporal fractional derivative on a graded mesh. In this case, they again employed a standard finite element discretization for the spatial diffusion term and used Newton's method to linearize the nonlinear driving term, achieving optimal-rate convergence of the solution in the \( L^\infty(H^1(\Omega)) \) space, given that $u_0 \in H_0^1(\Omega) \cap H^6(\Omega)$. In \cite{liao2020second}, a second-order, non-uniform time-stepping scheme is presented for problem (\ref{main}), utilizing FDM in the spatial direction. The authors demonstrated that the proposed scheme preserves the discrete maximum principle.
In \cite{MR4499606}, the author applied an adaptive L1 time-stepping scheme to solve problem (\ref{main}) and established a $L^2$-norm error estimate. In the study by Wang et al. \cite{wang2022local}, the authors utilized a nonuniform L1 scheme for the temporal direction and employed the LDG method for the spatial direction. They proved the \( L^2 \) stability and optimal error estimates for these two schemes using discrete fractional Gr\"{o}nwall-type inequalities. 

Beyond deterministic models, stochastic variants of the time-fractional Allen--Cahn equation have also been studied. In particular, Al-Maskari and Karaa \cite{MR4530741} analyzed the strong approximation of a stochastic time-fractional Allen–Cahn equation driven by additive fractionally integrated Gaussian noise. Their approach employed a piecewise linear finite element method for spatial discretization and a classical Grünwald–Letnikov approximation for the temporal discretization of both the Caputo fractional derivative and the associated fractional integral operator, with the stochastic forcing approximated via an 
$L^2$-projection. Table~\ref{Regularity_assumptions} gives a comprehensive comparison of the existing literature on $\alpha$-robust error estimates for fully discrete methods for the problem (\ref{main}). By $\alpha$-robustness, we mean that all estimates and constants are uniformly bounded as $\alpha$ approaches $1^-$.

\begin{table}[ht]
\centering
\begingroup
\footnotesize
\setlength{\tabcolsep}{3pt}
\renewcommand{\arraystretch}{1.3}

\begin{tabular}{|p{0.6cm} |p{4.1cm} |p{2.1cm}| p{1.9cm} |p{1.98cm}| p{1.07cm}|}
\hline
\textbf{Ref.} & \textbf{Regularity }& \textbf{Full} & \textbf{Nonlinearity} & \textbf{Initial} & \textbf{$\alpha$} \\
& \textbf{Requirement ($k\ge 1$)}& \textbf{Discretization} &  & \textbf{Condition} & \textbf{Robust} \\
\hline\hline

\cite{huang2020optimal} &
$\|u^{(k)}\|_{H^2(\Omega)} \le C(1+t^{\alpha-k})$ &
L1--FEM &
Newton linearization &
$H_0^1(\Omega)\cap H^4(\Omega)$ &
No \\

\cite{liao2020second} &
$\|u^{(k)}\|_{W^{2,\infty}(\Omega)} \le C(1+t^{\sigma-k})$, $\sigma\in(0,1)\cup(1,2)$ &
Alikhanov--FDM &
Implicit &
$H_0^1(\Omega)\cap H^4(\Omega)$ &
No \\

\cite{MR4499606} &
$\|u\|_{W^{4,\infty}}\le C$, $\|u^{(k)}\|_{W^{0,\infty}}\le C(1+t^{\sigma-k})$, $\sigma\in(0,1)\cup(1,2)$ &
Adaptive L1 &
Simple fixed-point algorithm &
$H_0^1(\Omega)\cap H^4(\Omega)$ &
No \\

\cite{wang2022local} &
$|u^{(k)}|\le C(1+t^{\alpha-k})$ &
L1--LDG, Alikhanov--LDG &
Implicit &
$H_0^1(\Omega)\cap H^4(\Omega)$ &
No \\

\cite{MR4402734} &
$\|u^{(k)}\|_{H^4(\Omega)} \le C(1+t^{\alpha-k})$ &
Alikhanov--FEM &
Newton linearization &
$H_0^1(\Omega)\cap H^6(\Omega)$ &
Yes \\

\hline
\end{tabular}

\caption{Relevant articles on fully discrete methods for the time-fractional Allen--Cahn equation and the minimum regularity requirement, where $u^{(k)}$ denotes the $k$th time derivative.}
\label{Regularity_assumptions}

\endgroup
\end{table}

 This paper aims to address this gap by deriving error estimates for a non-uniform Alikhanov mixed FEM applied to the problem (\ref{main}), and $u_0\in H_0^1(\Omega) \cap H^{3+\epsilon}(\Omega)$ with $0<\epsilon\leq1$. The main contributions of the present work are as follows:
\begin{itemize}
    \item New regularity results is discussed, with proofs provided in Section~\ref{section2}.
    
\item A modified discrete fractional Gr\"{o}nwall inequality is derived under a reasonable time-step restriction, which allows us to obtain optimal convergence of the proposed method.
     \item An optimal order error estimate $\mathcal{O}(\log_e(N)(h^{2} + N^{-2} ))$ (see, Theorem \ref{L2H1errortheorem_AM}) is derived, with respect to $L^2$-norm for the solution $u$ and its flux $\boldsymbol{\sigma}$ of the problem (\ref{main}) with under the assumption $u_0\in H_0^1(\Omega) \cap H^{3+\epsilon}(\Omega)$ with $0<\epsilon\leq1$, where $N+1$ denotes the number of grid points in the temporal direction and $h$ is the maximum diameter of finite elements.   
      \item $\alpha$-robustness is ensured, i.e., all estimates and bounds maintain validity as $\alpha$ approaches $1^-$.
      \item Numerical results to validate theoretical findings.
\end{itemize}

The subsequent sections of the paper are outlined as follows: In Section~\ref{section2}, well-posedness and regularity results are established for the problem (\ref{main}), along with some significant results for subsequent analysis. The variational formulation of the problem (\ref{main}) is introduced, and a non-uniform Alikhanov mixed FEM is proposed in Section~\ref{section3}. Error estimates of the proposed method are derived in Section~\ref{section4}. Theoretical findings are validated through numerical experiments in Section~\ref{section6}. Finally, the article concludes in Section~\ref{section7}.

In this paper, the symbols \( C \) and \( c \) consistently represent a positive generic constant. It is important to note that this constant may vary throughout the text and is independent of both the temporal grid size \( \Delta t \) and the spatial mesh size \( h \). Additionally, \( \mathbb{N}_0 \) denotes the set of nonnegative integers, which means \( \mathbb{N}_0 = \mathbb{N} \cap \{0\} \).

%
\section{Well-posedness of solutions and Regularity results}\label{section2}

%
\se 
Let $(\cdot,\cdot)$ denote the usual inner product in $L^2(\Omega)$ with induced norm  $\|\cdot\|$. We denote by $H^s(\Omega)$ the standard Sobolev space on $\Omega$
with the usual norm $\|\cdot\|_{H^s(\Omega)}$ for real $s$.
Let $u$ be a solution of problem  \eqref{main}. Multiplying the equation in \eqref{main} by $u_t$ and integrating over $\Omega$, we obtain 
$$
({\cal I}^{1-\alpha}u_t,u_t)+\frac{\kappa^2}{2}\frac{d}{dt}\|\nabla u\|^2+\frac{d}{dt}\int_\Omega F(u)\,dx=0.
$$
If $u_0\in H_0^1(\Omega)$, then  integration  in time over $(0,t)$  yields
\begin{align}\label{2.7}
\int_0^t({\cal I}^{1-\alpha}u_t(s),u_t(s))\,ds+V(u(t))= V(u_0), \qquad  0\leq t<\infty,
\end{align}
where $V(u)=\frac{\kappa^2}{2}\|\nabla u\|^2+\int_\Omega F(u)\,dx$ is the free energy functional. By the positivity property of the integral $\int_0^t({\cal I}^{1-\alpha}u_t(s),u_t(s))\,ds$ and
 the Sobolev embedding of $H^1(\Omega)$ into $L^{4}(\Omega)$, \eqref{2.7} immediately implies an a priori bound: If $u_0\in H^1_0(\Omega)$ with $\|\nabla u_0\|\leq R$, then there exits $R^\ast=C(R)>0$ such that
\begin{align}\label{2.8}
\|\nabla u(t)\|\leq R^\ast, \qquad  0\leq t<\infty.
\end{align}
This type of a priori bound is known to play a key role in establishing global existence of solutions in the classical case with $\alpha=1$ and also in the fractional-order case $0<\alpha<1$, see \cite{MR4530741}. 

Let $\{(\lambda_j,\phi_j)\}_{j=1}^\infty$ be the Dirichlet eigenpairs of $A:=-\Delta$ on $\Omega$,
with $\{\phi_j\}_{j=1}^\infty$ being an orthonormal basis in $L^2(\Omega)$. 
For  $r\geq -1$, we denote by  $\dot H^r(\Omega)\subset H^{-1}(\Omega)$ the Hilbert space consisting of the functions $v=\sum_{j=1}^\infty  \langle v,\phi_j\rangle\phi_j$, where $\langle \cdot,\cdot\rangle$ denotes the duality pairing between $H^{-1}(\Omega)$ and $H_0^1(\Omega)$, and it coincides with the $L^2$-inner product if the function $v\in L^2(\Omega)$. The induced 
 norm $ \|\cdot\|_{\dot H^r(\Omega)}$ is defined by
$\|v\|_{\dot H^r(\Omega)}^2=\sum_{j=1}^\infty \lambda_j^r \langle v,\phi_j\rangle^2.$ 
Thus, 
 $\|v\|_{\dot H^0(\Omega)}=\|v\|$ is the norm in $L^2(\Omega)$, 
$\|v\|_{\dot H^1(\Omega)}$ is the norm in $H_0^1(\Omega)$, and $\|v\|_{H_0^1(\Omega) \cap H^{2}(\Omega)}=\|A v\|$ is the equivalent norm in $H^2(\Omega)\cap H^1_0(\Omega)$ \cite{MR2249024}. Besides, it is easy to verify that $\|v\|_{\dot H^{-1}(\Omega)}=\|v\|_{ H^{-1}(\Omega)}$ is the norm in $H^{-1}(\Omega)$.
Note that the spaces $\dot H^r(\Omega)$, $r\geq -1$, form a Hilbert scale of interpolation spaces.

By means of Laplace transforms, the solution of problem \eqref{main} can be represented as (cf. \cite[Section 2]{al2019numerical})
\begin{align}\label{form-1}
u(t)=F(t)u_0+\int_0^t E(t-s)f(u(s))\,ds,\quad t>0,
\end{align}
where the solution operators $F(t):L^2(\Omega)\to L^2(\Omega)$ and $E(t):L^2(\Omega)\to L^2(\Omega)$ are defined by
$$
 F(t) = \frac{1}{2\pi i}\int_{\Gamma_{\theta,\delta}}e^{zt} z^{\alpha-1}(z^\alpha I+A)^{-1} \,dz,
\quad 
E(t) = \frac{1}{2\pi i}\int_{\Gamma_{\theta,\delta}}e^{zt}(z^\alpha I+A)^{-1} \,dz,
$$
with integration over a  contour $\Gamma_{\theta,\delta}\subset \mathbb{C}$ (oriented with an increasing imaginary part): 
$$
\Gamma_{\theta,\delta}=\{\rho e^{\pm i\theta}:\rho\geq \delta\}\cup\{\delta e^{i\psi}: |\psi|\leq \theta\},
$$
 for $\theta\in (\pi/2,\pi)$ and  $\delta> 0$. Then we have the following properties.
\begin{lemma}[{\cite{MR4290515}}]\label{LL}
Let $\mu\in [0,2]$. The operators $F(t)$ and $E(t)$ satisfy
\begin{itemize}
\item[\namedlabel{itmLL:a}{(a)}] $F(t):L^2(\Omega)\to \dot{H}^2(\Omega)$ is continuous with respect to $t\in (0,T]$,  $I-F(t)=\int_0^t AE (s)ds,\;$ and \\ $\;t^{k}\|A^{\mu/2}F^{(k)}(t)\|  \leq ct^{-\mu\alpha/2}\;$  and $\;t^{k+1}\|A^{-\mu/2}F^{(k+1)}(t)\|  \leq ct^{\mu\alpha/2}$, $k=0,1,2,\ldots,K$ for $K \in \mathbb{N}_0$.
\item[\namedlabel{itmLL:b}{(b)}] $E(t): L^2(\Omega)\to \dot{H}^2(\Omega)$ is continuous with respect to $t\in (0,T]$ and \\
$t^{1-\alpha}\|E(t)\|+ t^{2-\alpha}\|E'(t)\|+
t^{\alpha(\mu/2-1)+1}\|A^{\mu/2}E(t)\|\leq c  \; \forall t\in (0,T].$
\end{itemize}
\end{lemma}

We will frequently use the following lemma, which generalizes the classical Gronwall's inequality (see \cite{MR1122059}).
\begin{lemma}\label{Gronwall}
Assume that $y$ is a nonnegative function in $L^1(0,T)$ which satisfies
$$
y(t) \leq g(t)+\beta\int_0^t(t-s)^{-\alpha}y(s)\,ds\quad \mbox{ for } t\in (0,T],
$$
where $0\leq g(t)\in L^1(0,T)$, $\beta\geq 0$, and $0<\alpha<1$. Then there exists a constant $C_T$ such that 
$$
y(t) \leq g(t)+ C_T\int_0^t(t-s)^{-\alpha}g(s)\,ds\quad \mbox{ for }  
t\in (0,T].
$$
\end{lemma}



%
%

We now recall the following well-posedness result for problem~\eqref{main}.

\begin{theorem}\label{WP}\cite[Theorem 6.20]{MR4290515}
    If $u_0\in H_0^1(\Omega) \cap H^{2}(\Omega)$,  then the solution $u$ to problem  \eqref{main}  satisfies for any $\beta\in [0,1)$ and 
    $t\in J$
\begin{equation}\label{BG-1}
u\in C^{\alpha}(\bar{J};L^2(\Omega)) 
\cap   C(\bar{J};\dot{H}^{2}(\Omega)),\quad \partial_t^\alpha u \in C(\bar{J};L^{2}(\Omega));
\end{equation}
\begin{equation}\label{BG-2}
 Au\in C(J;H^{2\beta}(\Omega)),\quad \|A^{1+\beta}u(t)\|_{L^2(\Omega)}\leq ct^{-\beta\alpha};
\end{equation}
 \begin{equation}\label{BG-3}
 \partial_t u(t)\in C(J;H^{2\beta}(\Omega))\quad \text{ and } \quad \|A^{\beta}\partial_t u(t)\|_{L^2(\Omega)}\leq ct^{\alpha(1-\beta)-1},
 \end{equation}
where the constant c depends on $\|A u_0\|_{L^2(\Omega)}$  and $T$.
\end{theorem}
The proofs of these regularity results rely on the uniform bound 
$u\in L^\infty(\Omega\times J)$, see \cite[Theorem~6.19]{MR4290515}.
Also, in view of the identity
$
\Delta u^{3} = 6u|\nabla u|^{2} + 3u^{2}\Delta u,
$
the estimate $u \in C(\bar I; \dot H^{2}(\Omega))$, and the uniform bound 
$u\in L^\infty(\Omega\times J)$, there holds
\[
\|A f(u)\|_{L^{2}(\Omega)}
\le c\|\Delta u\|_{L^{2}(\Omega)}
+ c\|u|\nabla u|^{2}\|_{L^{2}(\Omega)}
+ c\|u^{2}\Delta u\|_{L^{2}(\Omega)}
\le c_T,
\]
with $c_{T}$ independent of $t$.


In what follows, we assume that 
$u_0\in H_0^1(\Omega)\cap H^{3+\varepsilon}(\Omega)$ with $0<\varepsilon\le 1$, 
which allows us to obtain improved regularity estimates, needed for the error analysis.

\begin{theorem}\label{Regularity_condition-b} 
Let $u_0\in H_0^1(\Omega) \cap H^{3+\epsilon}(\Omega), \; 0<\epsilon\leq1$.
Then,  problem  \eqref{main}  has a unique solution  $u$ satisfying for $t\in J$
\begin{equation*}
 \|A^{3/2}u(t)\|+ t^\alpha
\|A\partial^{\alpha}_t u(t) \| \leq c, \quad t^l\|\partial_{t}^{l}u(t)\|\leq C t^{\alpha} \quad \text{ and } \quad t\|A\partial_{t}u(t)\| \leq 
ct^{\alpha(1+\epsilon)/2}. 
\end{equation*}
Additionally, with $\boldsymbol{\sigma}=\nabla u$, we have
\begin{equation*}
    t\|A^{1/2}\partial_t\left(\nabla\cdot\boldsymbol{\sigma}(t)\right)\|\leq c t^{\epsilon\alpha/2},\qquad t^l\|\partial_{t}^{l}\boldsymbol{\sigma}(t)\|\leq c t^{\alpha},\qquad l =1,2,3. 
\end{equation*}
where the constant c may depend on $T$.
\end{theorem}


\begin{proof}
The estimate $\|A^{3/2}u(t)\| \leq c$ follows from \eqref{BG-2} and the regularity assumption on $u_0$.
Differentiating the representation (\ref{form-1}) with respect to $t$ and apply $A$ gives
 \begin{eqnarray*}\label{3i}
     Au^{\prime }(t) &=&  A^{-(1+\epsilon)/2}F^{\prime }(t)A^{(3+\epsilon)/2}u_0 + AE(t)f(u_0)+ \int_0^t AE(t-s)[(u'-3u^2u')(s)] ds\nonumber \\
     &=:&\sum_{i=1}^3I_i.
 \end{eqnarray*}
By Lemma~\ref{LL}, the terms $I_1$ and $I_2$ are respectively bounded by 
 \begin{equation*}\label{firstb11}
   \|I_1\|\leq ct^{\alpha(1+\epsilon)/2-1}\| A^{(3+\epsilon)/2} u_0\|\qquad \text{and} \qquad  
   \|I_2\|\leq  ct^{\alpha-1}\|Au_0\|.
 \end{equation*}
 The third term follows 
 \begin{equation*}\label{sb1}
     \|I_3\|\leq c \int_0^t (t-s)^{\alpha-1}\| Au^{\prime}(s)\|ds.
 \end{equation*}
The last three lemmas and the  Grönwall’s inequality in Lemma~\ref{Gronwall} give the desired estimate $t\|A\partial_{t}u(t)\| \leq ct^{\alpha(1+\epsilon)/2}$.

Now, multiply by $A^{3/2}$ instead of $A$ to get
\begin{eqnarray*}\label{3i-b}
    A^{3/2}u^{\prime}(t) &=&  A^{-\epsilon/2}F^{\prime }(t) A^{(3+\epsilon)/2}u_0 + 
     A^{1/2} E(t)Af(u_0)\nonumber \\
     &&+ \int_0^t A^{3/2}E(t-s)f'(u(s))u'(s)\,ds.
\end{eqnarray*}
Then, it follows that 
$$
\|A^{3/2}u^{\prime}(t)\|\leq ct^{\epsilon \alpha/2-1}\|A^{(3+\epsilon)/2}u_0\| +
 ct^{\alpha/2-1}\|Af(u_0)\| +c\int_0^t (t-s)^{\alpha/2-1}\|Au'(s)\|\,ds.
$$
Using previous bound for $\|Au'(s)\|$ and apply Gronwall's inequality yields
$$
\|A^{3/2}u^{\prime}(t)\|\leq ct^{\epsilon \alpha/2-1}\|A^{(3+\epsilon)/2}u_0\|,
$$
where $c$ depends on $T$. To establish the estimate $t^l\|\partial_t^l u(t)\|\le c t^\alpha$, we start from
\begin{align}\label{eq:Ah1}
     u'(t) =  A^{-1}F'(t)Au_0 +  E(t)f(u_0)+ \int_0^t E(t-s)f'(u(s))u'(s)\,ds .
\end{align}
Consequently,
$$
\|u^{\prime }(t)\| \leq ct^{\alpha-1} \|Au_0 \|+ ct^{\alpha-1} \|f(u_0)\|+c\int_0^t(t-s)^{\alpha-1}\|u'(s)\|\,ds.
$$
An application of Grönwall's inequality yields the desired estimate for $l=1$.

For higher derivatives, we proceed similarly. To remove the singular behavior near
$t=0$, we multiply \eqref{eq:Ah1} by $t$ and obtain
 \begin{align*}
    tu^{\prime }(t) =  &A^{-1}tF^{\prime }(t)Au_0 +  tE(t)f(u_0)+ \int_0^t  (t-s)E(t-s)f'(u(s))u'(s)\,ds\\
     &+\int_0^t  sE(t-s)f'(u(s))u'(s)\,ds=\sum_{i=1}^4J_i(t).
\end{align*}
We estimate the derivatives $J_i'(t)$ separately. By Lemma~\ref{LL},
\begin{align*}
     \|J_1'(t)+J_2'(t)\| = \| A^{-1}[(F'(t)+tF''(t)]Au_0 +  [E(t)+tE'(t)]f(u_0)\|\leq c_T t^{\alpha-1}. 
\end{align*}
Using again Lemma~\ref{LL} together with the bound already proved for $l=1$,
we obtain
\begin{align*}
      \|J_3'(t)\| = &\int_0^t  \| E(t-s)f'(u(s))u'(s)\|\,ds\\
     &+\int_0^t (t-s) \|E'(t-s)f'(u(s))u'(s)\|\,ds\\
     \leq &\; c\int_0^t(t-s)^{\alpha-1} s^{\alpha-1}\,ds\leq c_T t^{\alpha-1}.
\end{align*}
For the last term, we use that $tu'(t)\in L^\infty(\Omega\times J)$
(see \cite[Theorem~10.8]{MR4290515} for $\ell=1$), which implies
\begin{align*}
      \|J_4'(t)\| = &\int_0^t  \| E(t-s)[f'(u(s))u'(s)+sf''(u(s))(u'(s))^2]\|\,ds\\
     &+\int_0^t s\|E(t-s)f'(u(s))u''(s)\|\,ds\\
     \leq &\; c_T t^{\alpha-1}+c\int_0^t (t-s)^{\alpha-1} \|su''(s)\|\,ds.
\end{align*}
Combining the above bounds gives
\begin{align*}
     \|tu''(t)\| \leq c_T t^{\alpha-1}+c\int_0^t (t-s)^{\alpha-1} \|su''(s)\|\,ds,
\end{align*}
and another application of Grönwall’s inequality proves the desired estimate
for $l=2$. The case $l=3$ follows analogously and is therefore omitted. To establish the estimate $t^l\|A^{1/2}\partial_{t}^{l}u(t)\|\leq c t^{\alpha}$, we use the equation
\begin{align}\label{eq:Ah}
     A^{1/2}u^{\prime }(t) =  A^{-1}F^{\prime }(t)A^{3/2}u_0 +  A^{1/2}E(t)f(u_0)+ \int_0^t  A^{1/2}E(t-s)f'(u(s))u'(s)\,ds
\end{align}
and deduce that
$$
\|A^{1/2}u^{\prime }(t)\| \leq ct^{\alpha-1} \|A^{3/2}u_0 \|+ ct^{\alpha-1} \|A^{1/2}f(u_0)\|+c\int_0^t(t-s)^{\alpha-1}\|A^{1/2}u'(s)\|\,ds.
$$
The Gronwall's inequality yields the desired estimate for $l=1$.

 For the remaining cases, we use a similar approach. To avoid singularities near $t=0$, we first multiply \eqref{eq:Ah} by $t$ and obtain
 \begin{eqnarray*}
     A^{1/2}tu^{\prime }(t) &=&  A^{-1}tF^{\prime }(t)A^{3/2}u_0 \\
     &&+  A^{1/2}tE(t)f(u_0)+ \int_0^t  A^{1/2}(t-s)E(t-s)f'(u(s))u'(s)\,ds\\
     &&+\int_0^t  A^{1/2}sE(t-s)f'(u(s))u'(s)\,ds=\sum_{i=1}^4I_i(t).
\end{eqnarray*}
Below, we bound the derivatives $I_i'(t)$ separately. By Lemma~\ref{LL}, we immediately have
\begin{eqnarray*}
     \|I_1'(t)+I_2'(t)\| &=& \| A^{-1}[(F'(t)+tF''(t)]A^{3/2}u_0 +  A^{1/2}[E(t)+tE'(t)]f(u_0)\|\nonumber \\
     &\leq& c_T t^{\alpha-1}. 
\end{eqnarray*}
Again, using Lemma~\ref{LL} and the bound derived for $l=1$, $I_3'(t)$ can be bounded by
\begin{align*}
      \|I_3'(t)\| = &\int_0^t  \| E(t-s)A^{1/2}f'(u(s))u'(s)\|\,ds\\
     &+\int_0^t (t-s) \|E'(t-s)A^{1/2}f'(u(s))u'(s)\|\,ds\\
     \leq &\; c\int_0^t(t-s)^{\alpha-1} s^{\alpha-1}\,ds\leq c_T t^{\alpha-1}.
\end{align*}
To bound \( I_4'(t) \), we can follow a similar argument as that used for \( J_4'(t) \) and arrive at 
\begin{align*}
      \|I_4'(t)\| = &\int_0^t  \| E(t-s)A^{1/2}[f'(u(s))u'(s)+sf''(u(s))(u'(s))^2]\|\,ds\\
     &+\int_0^t s\|E(t-s)A^{1/2}f'(u(s))u''(s)\|\,ds\\
     \leq &\; c_T t^{\alpha-1}+c\int_0^t (t-s)^{\alpha-1} \|A^{1/2}su''(s)\|\,ds.
\end{align*}
Combining the preceding estimates, we finally  obtain
\begin{align*}
     \|A^{1/2}tu''(t)\| \leq c_T t^{\alpha-1}+c\int_0^t (t-s)^{\alpha-1} \|A^{1/2}su''(s)\|\,ds,
\end{align*}
which proves the desired estimate with $l=2$. For the remaining case $l=3$, the arguments are nearly identical and so not shown.

Finally, in order to establish the estimate $\;t^{\alpha}\|A\partial^{\alpha}_t u \| \leq c$, we appeal to the identity $$t\partial^{\alpha}_t u = \partial^{\alpha}_t (tu) - \alpha \int_0^t \omega_{1-\alpha}(t-s)u(s)\;ds - t\omega_{1-\alpha}(t)u (0), $$ and the fact that 
$\|A u(t)\|+t\|A\partial_t u\|\leq c_T$, which allows us to derive 
    \begin{align}
        \nonumber  \|At\partial^{\alpha}_t u\| &\leq \|A\partial^{\alpha}_t (tu)\| + \alpha \int_0^t \omega_{1-\alpha}(t-s)\|Au\|\;ds + t\omega_{1-\alpha}(t)\|Au (0) \|\\
        &\nonumber \leq \int_0^t \omega_{1-\alpha}(t-s)\|A u+sA\partial_s u\|\,ds + c\frac{\alpha t^{1-\alpha}}{\Gamma{(2-\alpha)}} + c\frac{t^{1-\alpha}}{\Gamma{(1-\alpha)}}\\
        \nonumber &\leq c_Tt^{1-\alpha}.
    \end{align}
which completes the proof.
\end{proof}

%
\section{Mixed Formulation and Full discretization}\label{section3}
%

By following the procedure outlined in \cite{tripathi2025non}, we introduce a new variable flux defined as $\boldsymbol{\sigma} = \nabla u$. The problem described in (\ref{main}) can be rewritten as follows:
\begin{align*}
& \boldsymbol{\sigma} - \nabla u =0\quad\text{in}\quad\Omega\times J,\\
&\partial^{\alpha}_t u - \kappa^2\nabla \cdot \boldsymbol{\sigma} =\;f(u) \quad\text{in}\quad\Omega\times J,\\
& u = 0 \quad\text{on}\quad\partial\Omega\times J,\\
& u(\cdot,0) = u_0(\cdot) \quad\text{in}\quad\Omega.
\end{align*}
Thus, a mixed formulation of the problem (\ref{main}) with $V:=L^2(\Omega)$, and  $\boldsymbol{W}:=H(div; \Omega)$ is to find the solution pair $(u,\boldsymbol{\sigma}):J\to V\times \boldsymbol{W}$, such that $u(0)=u_0$ and 
\begin{align}
\label{variation1MFEM} (\boldsymbol{\sigma}, \boldsymbol{w}) + ( u,  \nabla \cdot \boldsymbol{w})~ &=~0 \quad \forall \boldsymbol{w} \in \boldsymbol{W},\; t\in J,\\
    \label{variation2}\langle\partial_t^{\alpha} u,v \rangle  - \kappa^2(\nabla \cdot \boldsymbol{\sigma},v) ~&=~ \langle f(u),v\rangle  \quad \forall v \in V,\; t\in J.
\end{align}

 \textit{Spatial discretization:} Let $\mathcal{T}_h$ be a regular family of decomposition of $\Omega$ (see \cite{Ciarlet1978}) into closed d-simplexes $Q$ of size $h = \max\{\text{diam}(Q); Q \in \mathcal{T}_h\}$.  Further, let $V_h\times\boldsymbol{W}_h$ be a finite element subspace of $V\times\boldsymbol{W}$ having the following three properties:
\begin{enumerate}
    \item[(i)] $\nabla\cdot \boldsymbol{W}_h \subseteq V_h$,
    \item[(ii)] There exists a projection $\Pi_h : \boldsymbol{W} \rightarrow \boldsymbol{W}_h$, called Fortin projection satisfying $\nabla \cdot \Pi_h = P_h(\nabla \cdot)$. Here, $P_h :V \rightarrow V_h$ is the $L^2$-projection defined by $(P_h v -v,v_h)=0\; \forall~ v_h \in V_h, \; v \in V$, and hence 
    \begin{eqnarray*}
    (\nabla \cdot (\Pi_h \boldsymbol{w} - \boldsymbol{w}  ), v_h) &=&0 \; \forall~ v_h \in V_h,
\;
\text{ and }\;\\
(P_h v-v, \nabla \cdot \boldsymbol{w}_h) &=&0 \; \forall~ \boldsymbol{w}_h \in \boldsymbol{W}_h,
\end{eqnarray*}
    \item[(iii)] Approximation properties:
    \begin{eqnarray}\label{approxprop}
        \|\boldsymbol{w}- \Pi_h \boldsymbol{w} \| &\leq& C h^r \| \nabla \cdot \boldsymbol{w}\|_{r-1};\nonumber \\ \quad \|v -P_h v\| + h \|v -P_h v\|_1 &\leq& Ch^r\|v\|_r,\;\qquad1\leq r \leq 2.
    \end{eqnarray}
\end{enumerate}
Examples of such finite-dimensional subspaces $V_h\times\boldsymbol{W}_h$ of $V\times\boldsymbol{W}$ having the above properties can be found in \cite{MR483555} and \cite{MR1115205}.

\noindent\textit{Temporal Discretization:} Consider a partition $\{t_n\}_{n=0}^N$ of the interval $\bar{J}=[0, T]$ such that $t_n = \left(\frac{n}{N}\right)^{\gamma}T,\quad \gamma\geq 1,\;0\leq n \leq N$
and define a fractional time level $t_{n-\nu}:=\nu t_{n-1} + (1-\nu)t_n$ for an off-set parameter $\nu \in [0,1/2)$ and $\Delta t_j = t_j -t_{j-1}, ~ 1 \leq j \leq N,$ and we define local time-step ratio $ \mu_{j}=\Delta t_{j}/\Delta t_{j-1}, \quad  2 \leq j \leq N$ with  $\mu:= \max_{2\leq j \leq N}\mu_j.$ For any time sequence $\{\phi^j\}_{j=0}^N$, define the backward difference $\nabla \phi^j := \phi^j - \phi^{j-1}$. Also set $\phi^{n-\nu} := \nu \phi^{n-1} + (1-\nu)\phi^n$. Let $\Pi_{1,j}\phi$ denote the linear interpolant of a function $\phi$ with respect to the nodes $t_{j-1}$ and $t_j$, and let $\Pi_{2,j}\phi$ denote the quadratic interpolant with respect to $t_{j-1}$, $t_j$ and $t_{j+1}$. 
Note that
\begin{align*}
({\Pi}_{1,j}\phi)^{\prime}(t) = \frac{\nabla \phi^j}{\Delta t_j}  \text{ and }  ({\Pi}_{2,j}\phi)^{\prime}(t) = \frac{\nabla \phi^j}{\Delta t_j} + \frac{2(t-t_{j-1/2})}{\Delta t_j(\Delta t_j + \Delta t_{j+1})}\left(\frac{\nabla \phi^{j+1}}{\mu_{j+1}} - \nabla \phi^j\right).
\end{align*}
Now, the corresponding fully discrete mixed finite element approximation is to find a pair $(u_h^{n-\nu}, \boldsymbol{\sigma}^{n-\nu}_h)\in V_h \times \boldsymbol{W}_h$, $0 \leq n \leq N$ such that $u_h(0)=P_h u_{0}$ and
\begin{eqnarray}
\label{fullydiscrete1_AM} (\boldsymbol{\sigma}^{n-\nu}_h, \boldsymbol{w}_h) + ( u_h^{n-\nu}, \nabla \cdot  \boldsymbol{w}_h) &=&0 \quad \forall~ \boldsymbol{w}_h \in \boldsymbol{W}_h,\;0\leq n\leq N,\\
    \label{fullydiscrete2_AM}(D_{t_{n-\nu}}^{\alpha} u_h^{n-\nu},v_h) -(\nabla \cdot \boldsymbol{\sigma}_h^{n-\nu},v_h) &=& (u_h^{n-\nu} - P_h (u_h^{n-1})^3, v_h) \nonumber \\
    &&\;- 3(1-\nu)(P_h((u_h^{n-1})^2( u_h^{n} - u_h^{n-1})), v_h),
\end{eqnarray}
for al $v_h\in V_h$ and $1\leq n\leq N$.

Moreover, the well-known Alikhanov formula to approximate the Caputo derivative $\partial_{t}^{\alpha} \phi(t_{n-\nu})$ is given by
\begin{align}
    &D^{\alpha}_{t_{n-\nu}} \phi(t_{n-\nu}) := {\sum_{j=1}^{n}}K^{n,j}_{1-\alpha} \left(\phi(t_j)-\phi(t_{j-1})\right),\label{Alikhanovscheme} 
\end{align}
where 
\begin{align}
\label{discretekernel} K^{n,j}_{1-\alpha} &:= \begin{cases}
a^{1,1}_{1-\alpha}&\;\text{ for } j=1, n= 1,\\
a_{1-\alpha}^{n,n}+ \frac{1}{\mu_n}b_{1-\alpha}^{n,n-1} &\; \text{ for } j=n, n\geq 2,\\
a_{1-\alpha}^{n,j} + \frac{1}{\mu_j}b_{1-\alpha}^{n,j-1}-b_{1-\alpha}^{n,j} &\; \text{ for } 2 \leq j \leq n-1, n\geq 2,\\     
a_{1-\alpha}^{n,1}-  b_{1-\alpha}^{n,1} &\; \text{ for }  j=1, n\geq 2,
   \end{cases}
   \end{align}
   with the discrete coefficients $a_{1-\alpha}^{n,j}$ and $b_{1-\alpha}^{n,j}$ are defined by 
   \begin{align}
   \label{discretecoefficienta}a_{1-\alpha}^{n,j} & :=\frac{1}{\Delta t_j}\int_{t_{j-1}}^{\min\{t_{n-\nu},t_{j}\}}k_{1-\alpha}(t_{n-\nu}-s)\;ds \quad \text{for } 1 \leq j \leq n,\\
   b_{1-\alpha}^{n,j} & :=  \frac{1}{\Delta t_j(\Delta t_j + \Delta t_{j+1})}\int_{t_{j-1}}^{t_{j}}(s-t_{j-1})(t_j -s)\partial_s k_{1-\alpha}(t_{n-\nu}-s)\;ds\label{discretecoefficientb}
   \end{align}
for $1 \leq j \leq n-1$.
 
If $\nu=\alpha/2$ then in the limit $\alpha \rightarrow 1$,
this scheme reduces to the well-known Crank-Nicolson method ($\nu \rightarrow 1/2$). Therefore, the time-stepping scheme (\ref{Alikhanovscheme}) is also referred as a fractional
Crank–Nicolson method.

The following result will frequently be used in the convergence analysis of the proposed method.
\begin{lemma}\label{dfdoi}
   [\cite{MR3904430}, Lemma 4.1] The discrete Caputo formula (\ref{Alikhanovscheme}) satisfies 
\begin{align}\label{dfdoestimate}
\left((D^{\alpha}_{t_{n-\nu}}v)^{n-\nu},v^{n-\nu}\right) \geq \frac{1}{2}D^{\alpha}_{t_{n-\nu}}\|v^{n-\nu}\|^2,\;1\leq n\leq N,\quad \forall \nu \in [0,1/2).
\end{align} 
\end{lemma}
For the non-linear term $u^3(t_{n-\nu})$, we shall use the Newton linearization
\begin{align*}
    (u^{n-\nu})^3 \approx (u^{n-1})^3 +  3(u^{n-1})^2(u^{n-\nu} - u^{n-1}).
\end{align*}
\begin{lemma}\label{trun_l}
    The truncation error $\mathcal{E}^{n-\nu}:= (u^{n-\nu})^3  - (u^{n-1})^3 -  3(u^{n-1})^2(u^{n-\nu} - u^{n-1})$ of Newton linearization satisfies
    \begin{align*}
        \|\mathcal{E}^{n-\nu}\| \leq C n^{- \min\{\epsilon\gamma\alpha,\;2 \} } \; \text{for}\; 1\leq n \leq N.
    \end{align*}
\end{lemma}
\begin{proof}
By rearranging, we have $\mathcal{E}^{n-\nu}:= (u^{n-\nu})^3  - (u^{n-1})^3 -  3(u^{n-1})^2(u^{n-\nu} - u^{n-1})= (u^{n-\nu} - u^{n-1})^2(u^{n-\nu} + 2u^{n-1})$ for $1\leq n \leq N$. Now,
\begin{eqnarray*}
  \|\mathcal{E}^{n-\nu}\| &\leq& \|(u^{n-\nu} + 2u^{n-1})\|\; \|(u^{n-\nu} - u^{n-1})\|^2_{L^\infty}\\
  &\leq&  C \left( \int_{t_{n-1}}^{t_{n-\nu}} \|A u_t \|\;dt\right)^2.
\end{eqnarray*}
Using Theorem~\ref{Regularity_assumptions}, we get
\begin{align*}
    \|\mathcal{E}^{1-\nu}\| \leq C\left( \int_{t_0}^{t_{1-\nu}}t^{\epsilon\alpha/2 -1} dt\right)^2 \leq Ct_{1}^{\epsilon\alpha} \leq C N^{-\epsilon\gamma \alpha}.
\end{align*}
For $2 \leq n \leq N$, we have,
\begin{eqnarray*}
    \|\mathcal{E}^{n-\nu}\| &\leq&\; C\left( \int_{t_{n-1}}^{t_{n-\nu}}t^{\epsilon\alpha/2 -1} dt\right)^2 \\
    &\leq& C \left(\Delta t_n t_{n-1}^{\epsilon\alpha/2-1}\right)^2
    \leq C\left( N^{-\epsilon\gamma\alpha}(n-1)^{\epsilon\gamma\alpha-2}\right)\leq Cn^{-\min\{\epsilon\gamma\alpha,\;2\}}.
\end{eqnarray*}
\end{proof}

The following properties of discrete kernels $K_{1-\alpha}^{n,j}$ and their complementary discrete kernel
\begin{align}\label{complementarydiscretekernel}
&P_{\alpha}^{n,i} := \frac{1}{K_{1-\alpha}^{i,i}}\begin{cases}  
\displaystyle{\sum_{j=i+1}^n} P_{\alpha}^{n,j} \left(K_{1-\alpha}^{j,i+1} - K_{1-\alpha}^{j,i} \right) & \text{ : }  ~ 1 \leq i \leq n-1,\\
1 & \text{ : }  ~  i = n
\end{cases}
\end{align}
serve as key tools for deriving the subsequent error estimates.

\begin{lemma}\label{discretekernelproperties}
The discrete kernels $K_{1-\alpha}^{n,j}$ and $P_{\alpha}^{n,j}$ with $\pi_A=\frac{11}{4}$ and $\nu\in[0,\frac{1}{2})$ satisfy the following results:
\begin{enumerate} 
  \item[\namedlabel{itm:a}{(a)}] $0\; <\; P_{\alpha}^{n,j} \leq \pi_A \Gamma(2-\alpha) \Delta t_j^{\alpha}$ and $0<K_{1-\alpha}^{n,i-1} < K_{1-\alpha}^{n,i} \;$ for $1\leq j\leq n$ and $2\leq i \leq n$.
 \item[\namedlabel{itm:b}{(b)}] $\displaystyle{\sum_{j=i}^n P_{\alpha}^{n,j} K^{j,i}_{1-\alpha} = 1, ~ 1 \leq i \leq n}$.
  \item[\namedlabel{itm:c}{(c)}] $\displaystyle{\sum_{j=1}^{n} P_{\alpha}^{n,j}k_{1+(m-1)\alpha}(t_j) \leq \pi_A k_{1+m\alpha}(t_{n})}$ for any non-negative integer $0\leq m\leq \left\lfloor\frac{1}{\alpha} \right\rfloor$.
 \item[\namedlabel{itm:d}{(d)}] $ \nu \sum_{j=1}^{n-1}P_{\alpha}^{n,j}E_{\alpha}(\nu t_j^{\alpha}) \leq \pi_A (E_{\alpha}(\nu t_n^{\alpha})-1)$  for any constant $ \nu>0,$ provided $\Delta t_{n-1}\leq \Delta t_{n},\; n \geq 2$, where $\displaystyle{E_{\alpha}(z) := \sum^{\infty}
_{j=0}\frac{z^j}{\Gamma(j\alpha + 1)}}$ is the Mittag-Leffler function.
 \item[\namedlabel{itm:e}{(e)}]  $\displaystyle{\sum_{j=1}^{n}P_{\alpha}^{n,j}\;t_j^{\beta-\alpha}\;\leq \; \frac{\pi_A \Gamma{(1+\beta - \alpha)}}{\Gamma(1+\beta)}\;t_n^\beta\quad \forall \beta\in (0,1).}$
\item[\namedlabel{itm:f}{(f)}] $\sum_{j=1}^{n}P_{\alpha}^{n,j}\;t_j^{-\alpha}\;\leq \; 4 \pi_A e^\gamma\;\log_e(n+2), \;\text{ provided }\; t_n=\left(\frac{n}{N}\right)^\gamma T,\;\gamma\geq 1.$
    \item[\namedlabel{itm:g}{(g)}] $P^{n,n}_\alpha/ P^{n,n-1}_\alpha\;\leq\;\displaystyle\frac{2-\alpha}{\alpha}\mu_n^\alpha,\;n\geq 2.$ 
\end{enumerate}
\end{lemma}

\begin{proof}
The proof of \ref{itm:a} and \ref{itm:b} are provided in [\cite{MR3904430}~Lemma~2.1(1), \cite{MR4270344}~Theorem~2.1(II)]. By substituting $v(t) = k_{1+m\alpha}(t)$ and $k_{1+\beta}(t)$ into [\cite{MR3904430} Lemma~2.1~(2)], we obtain results \ref{itm:c} and \ref{itm:e}, respectively. The result \ref{itm:d} follows from [\cite{MR3904430}~Lemma~2.3]. The estimate \ref{itm:f} follows from \ref{itm:e}; however, for the sake of completeness and the reader's convenience, the proof is provided below. Define $\delta_n := \frac{1}{\log_e (n+2)}$ and consider
    \begin{align}\label{proplogfact}
         \sum_{j=1}^{n}P_{\alpha}^{n,j}\;t_j^{-\alpha}\;&=\; \sum_{j=1}^{n}P_{\alpha}^{n,j}\;t_j^{\delta_n-\alpha}\;t_j^{-\delta_n}\;\leq \; t_1^{-\delta_n}\;\sum_{j=1}^{n}P_{\alpha}^{n,j}\;t_j^{\delta_n-\alpha}.
    \end{align}
    Thus, an application of the estimate \ref{itm:e} with $\beta=\delta_n$ in (\ref{proplogfact}) yields the estimate \ref{itm:f} as follows
    \begin{eqnarray*}
         \sum_{j=1}^{n}P_{\alpha}^{n,j}\;t_j^{-\alpha}\;&\leq&\; \frac{\pi_A \Gamma(1+\delta_n -\alpha)}{\Gamma(1+\delta_n)} \;t_1^{-\delta_n}\;t_n^{\delta_n}\;\\
         &=&\; \frac{\pi_A \Gamma(2+\delta_n -\alpha)}{(1+\delta_n -\alpha)\Gamma(1+\delta_n)} \;n^{\gamma\delta_n}\;\leq\; 4 \pi_A e^\gamma\log_e(n+2) ,
    \end{eqnarray*}    
    where we have used $\Gamma(1+\delta_n)\geq 1/2$, $\Gamma(2+\delta_n -\alpha)\leq 2$, $1+\delta_n -\alpha\geq \delta_n$ and $n^{\gamma\delta_n} \leq e^\gamma$ to obtain the last term in the above estimate. To establish result \ref{itm:g}, we utilize the complementary relation (\ref{complementarydiscretekernel}) and the result  $K_{1-\alpha}^{n,n-1}/K_{1-\alpha}^{n,n} < (1-2\sigma)/(1-\sigma)$ from [\cite{MR4270344}~Theorem~2.1(III)], for $\sigma=\alpha/2$, to derive
     \begin{align}
        & P^{n,n-1}_\alpha = P^{n,n}_\alpha \;\left(1-\frac{K^{n,n-1}_{1-\alpha}}{K^{n,n}_{1-\alpha}}\right)/\mu_n^\alpha\;\geq \; P^{n,n}_\alpha \;\left(\frac{\alpha}{2-\alpha}\right)/\mu_n^\alpha.\label{lemma3.2additional3}
    \end{align}
The estimate (\ref{lemma3.2additional3}) yields the result \ref{itm:g}, thus completing the remainder of the proof.

\end{proof}

Previously established discrete fractional Gr\"{o}nwall inequalities, as discussed in \cite{MR4402734, MR3904430}, impose severe time-step restrictions in the temporal direction for (\ref{main}). 
In order to ease the severe time-step restriction to a milder one, we present below a refined discrete fractional 
Gr\"{o}nwall inequality, which allows us to prove the stability and convergence of the proposed Alikhanov-MFEM on a graded mesh for (\ref{main}), under a less restrictive time-step condition. The proof for Lemma \ref{DFGI} can be found in Appendix \ref{appendix_sec_gronwall_b}.

   \begin{lemma}\label{DFGI}{(Discrete fractional Gr\"{o}nwall inequality).}
    Let $\{v^{n}\}_{n=0}^{N}$, $\{\xi^{n}\}_{n=1}^{N}$, $\{\eta^{n}\}_{n=1}^{N}$ and $\{\zeta^{n}\}_{n=1}^{N}$ be non-negative finite sequences such that 
\begin{align}
    \label{dfgi1} 
     D^{\alpha}_{t_{n-\nu}}(v^{n-\nu})^2\;&\leq \; \sum_{i=0}^{n}\lambda_{i}^{n}(v^{i})^{2} + v^{n-\nu}\xi^{n} + (\eta^{n})^{2} +  (\zeta^{n})^{2},\quad 1\leq n \leq N,
\end{align}
where $\lambda_{j}^{n}\;\geq\; 0,\;0\leq j\leq n,$ and the discrete fractional differential operator $D^{\alpha}_{t_{n-\nu}},\; 1 \leq n\leq N,$ is given by (\ref{Alikhanovscheme}). If there exists a constant $\Lambda\;>\;0,$ such that, $\sum_{j=0}^{n}\lambda_{j}^{n}\;\leq\;\Lambda,\;1\leq n \leq N,$ and if $\Delta t_{n-1}\leq\Delta t_{n},\;2\leq n \leq N,$ with the maximum time-step size 
\begin{align} \label{timecondition}
\Delta t:=\max_{1\leq n \leq N} \Delta t_{n}\;\leq\; \left( \delta \pi_A  \Gamma{(2-\alpha) \max_{1 \leq n \leq N}\lambda_n^n}\right)^{-\frac{1}{\alpha}}, \text{ for some } \delta>1,
\end{align}
then, for $C_\delta:=\frac{\delta}{\delta-1} $ and $1\leq n \leq N$,  
\begin{align}
    \label{dfgi2} 
    v^{n}\;\leq &\; C_\delta E_{\alpha}(C_\delta \pi_A\Lambda t_{n}^{\alpha})\Bigg(v^{0} + \max_{1\leq j \leq n}\sum_{i=1}^{j}P_{\alpha}^{j,i} \xi^{i} + \left(2\pi_A t_n^\alpha \right)^{\frac{1}{2}}\max_{1\leq j \leq n}\eta^{j} \nonumber \\
    &\;+ \max_{1\leq j \leq n}\left(\sum_{i=1}^{j}P_{\alpha}^{j,i} (\zeta^{i})^2\right)^{\frac{1}{2}}\Bigg).
\end{align}

\end{lemma}
The following bound for the numerical solution \( u_h^n \) is obtained from this source.

\begin{lemma}
    
[\cite{MR3954448}, Theorem 3.2] The fully discrete Alikhanov-FEM has a unique solution
$u^{n}_h$ for $n = 1,..., N$, and there exists a positive constant $N^*$ such that when $N > N^*$ , one
has
\begin{align*}
\max_{1\leq n \leq N} \| u^{n}_h \|_{L^\infty(\Omega)} \leq C_3,
\end{align*}
where the fixed constant $C_3$ is independent of the mesh. 
\end{lemma}
\section{Error analysis}\label{section4}
This section presents auxiliary results and establishes optimal error estimates. At any temporal grid point \(t_{n-\nu}\), the variational problem \((\ref{variation1MFEM}--\ref{variation2})\) implies, for \(1 \leq n \leq N\),
\begin{align}
\label{variationalatt_n1_AM} 
&(\boldsymbol{\sigma}(t_{n-\nu}), \boldsymbol{w}_h) + ( u(t_{n-\nu}),  \nabla \cdot \boldsymbol{w}_h)~ =~0 \quad \forall~ \boldsymbol{w}_h \in \boldsymbol{W}_h,\\
\label{variationalatt_n2_AM}&(D_{t_{n-\nu}}^{\alpha}u(t_{n-\nu}),v_h)-\kappa^2(\nabla \cdot \boldsymbol{\sigma}(t_{n-\nu}),v_h) = (f(u(t_{n-\nu})),v_h)\\
&\hspace{5cm}+ (\Upsilon^{n-\nu}, v_h)\quad \forall\;v_h \in V_h.   
\end{align}
By subtracting equations (\ref{fullydiscrete1_AM}--\ref{fullydiscrete2_AM}) from (\ref{variationalatt_n1_AM}--\ref{variationalatt_n2_AM}), we obtain the following error equation, for $1\leq n \leq N$:
\begin{align}
\label{errorequation1_AM}
&( e_{\boldsymbol{\sigma}}^{n-\nu}, \boldsymbol{w}_h) + (e_u^{n-\nu}, \nabla  \cdot \boldsymbol{w}_h)=\;0 \quad \forall~ \boldsymbol{w}_h \in \boldsymbol{W}_h \\
\label{errorequation2_AM}
&(D_{t_{n-\nu}}^{\alpha}e^{n-\nu}_u,v_h) - \kappa^2(\nabla \cdot e_{\boldsymbol{\sigma}}^{n-\nu},v_h) =\; (\Upsilon^{n-\nu} + f(u(t_{n-\nu}))- u_h^{n-\nu}  \nonumber \\
&\hspace{2cm}+ P_h ( (u_h^{n-1})^3 ) + 3(1-\nu)P_h((u_h^{n-1})^2( u_h^{n} - u_h^{n-1})), v_h) \quad \forall\;v_h \in V_h,   
\end{align}
where $\Upsilon^{n-\nu}:= (D^{\alpha}_{t_{n-\nu}}u(t_{n-\nu}) - \partial^{\alpha}_{t}u(t_{n-\nu}))$. Also, $e^{n-\nu}_u:= u(t_{n-\nu})-u_h^{n-\nu}$ and $e^{n-\nu}_{\boldsymbol{\sigma}}:= \boldsymbol{\sigma}(t_{n-\nu})-\boldsymbol{\sigma}_h^{n-\nu}$ denote the error between the exact solution $u(t_{n-\nu})$ and the approximate solution $u_h^{n-\nu}$ and between the flux $\boldsymbol{\sigma}(t_{n-\nu})$ and its approximation $\boldsymbol{\sigma}_h^{n-\nu}$, respectively, at time level $t = t_{n-\nu}$. To derive optimal error estimates, we further decompose the errors $e^{n-\nu}_u$ and $e^{n-\nu}_{\boldsymbol{\sigma}}$ further as follows:
\begin{align*}
    &e^{n-\nu}_u = \eta^{n-\nu} + \theta^{n-\nu}, \; \eta^{n-\nu} = u(t_{n-\nu})- P_h u(t_{n-\nu}), \; \theta^{n-\nu} =  P_h u(t_{n-\nu}) - u_h^{n-\nu} ,\\
    &e^{n-\nu}_{\boldsymbol{\sigma}} =   \boldsymbol{\zeta}^{n-\nu} + \boldsymbol{\xi}^{n-\nu}, \; \boldsymbol{\zeta}^{n-\nu} = \boldsymbol{\sigma}(t_{n-\nu}) - \Pi_h \boldsymbol{\sigma}(t_{n-\nu}) , \; \boldsymbol{\xi}^{n-\nu} =  \Pi_h \boldsymbol{\sigma}(t_n)  - \boldsymbol{\sigma}_h^{n-\nu}.
\end{align*}
 Since the estimates for the projection errors $\eta^{n-\nu}:=u(t_{n-\nu})- P_h u(t_{n-\nu}) $ and $\boldsymbol{\zeta}^{n-\nu}:=\boldsymbol{\sigma}(t_{n-\nu}) - \Pi_h \boldsymbol{\sigma}(t_{n-\nu}) $ is already known (see (\ref{approxprop})), establishing the final error estimate requires only the estimation of $\theta^{n-\nu}$ and $\boldsymbol{\xi}^{n-\nu}$. Using (\ref{errorequation1_AM}-\ref{errorequation2_AM}), $\theta^{n-\nu}$ and $\boldsymbol{\xi}^{n-\nu}$ satisfy the following relations:
 \begin{align}
     \label{theta1_AM}&({\boldsymbol{\xi}}^{n-\nu} + {\boldsymbol{\zeta}}^{n-\nu}, \boldsymbol{w}_h) + ( \theta^{n-\nu},  \nabla \cdot \boldsymbol{w}_h)~ =~0 \quad \forall~ \boldsymbol{w}_h \in \boldsymbol{W}_h, \\
\label{theta2_AM}
&(D_{t_{n-\nu}}^{\alpha}\theta^{n-\nu},v_h) - \kappa^2(\nabla \cdot {\boldsymbol{\xi}}^{n-\nu},v_h)  =-(D_{t_{n-\nu}}^{\alpha}\eta^{n-\nu},v_h) +(\Upsilon^{n-\nu} - \mathcal{E}^{n-\nu} \nonumber\\
&\hspace{1cm} + (1-\nu-\phi_{n,2})e_u^{n} + (\nu - \phi_{n,1})e^{n-1}_u, v_h)   \quad \forall\;v_h \in V_h, \; 1\leq n \leq N.  
 \end{align}

\begin{lemma}\label{thetaestimate_l2norm_AM}
If for some $\delta>1$, $\max\limits_{1\leq n \leq N}\Delta t_n\leq \left( \delta \pi_A L(1-\nu)^2\Gamma{(2-\alpha)}\right)^{-\frac{1}{\alpha}}$, then $\theta^{n}$ satisfies, for $1 \leq n \leq N$,
\begin{align}\label{L2errorestimateeqn1_AM}
 &\|\theta^{n}\| \leq   C_\delta E_{\alpha}(C_\delta \pi_A\Lambda t_{n}^{\alpha})  \Bigg( \|\theta^0\|  + 2\max_{1\leq j \leq n} \sum_{i=1}^j P_{\alpha}^{j,i}\Big(\|\Upsilon^{i-\nu} \| + \| \mathcal{E}^{i-\nu}\| +\| D^{\alpha}_{t_{i-\nu}}\eta^{i-\nu}\|   \Big) \nonumber \\
 &\hspace{2cm}+ 3\max\{9C_3^4,C_4^2 \nu^{-2}\}\max_{1\leq j \leq n} \left(\sum_{i=1}^j P_{\alpha}^{j,i}\|\eta^{i-\nu}\|^2\right)^{1/2}  \nonumber\\
 &\hspace{2cm} + \kappa^{-2}\max_{1\leq j \leq n}\left( \sum_{i=1}^{j}P_{\alpha}^{j,i} \|\boldsymbol{\zeta}^{i-\nu}\|^2 \right)^{1/2}  \Bigg),
 \end{align}
 where $L= 6 + 27C_3^4$  and $\Lambda = 6\nu^2 + 3C_4^2$.
\end{lemma}
\begin{proof}
Substitute $\boldsymbol{w}_h = \boldsymbol{\xi}^{n-\nu}$ and $v_h = \theta^{n-\nu}$ in (\ref{theta1_AM}) and (\ref{theta2_AM}), respectively. Then, apply (\ref{theta1_AM}) in (\ref{theta2_AM}) to obtain
\begin{align*}
(D_{t_{n-\nu}}^{\alpha}\theta^{n-\nu},\theta^{n-\nu}) &+ \kappa^{-2}( {\boldsymbol{\xi}}^{n-\nu}, \boldsymbol{\xi}^{n-\nu}) = - \kappa^{-2}( {\boldsymbol{\zeta}}^{n-\nu}, \boldsymbol{\xi}^{n-\nu}) -(D_{t_{n-\nu}}^{\alpha}\eta^{n-\nu},\theta^{n-\nu}) \\
&+ (\Upsilon^{n-\nu} - \mathcal{E}^{n-\nu} + (1-\nu-\phi_{n,2})e_u^{n} + (\nu - \phi_{n,1})e^{n-1}_u, \theta^{n-\nu}).
\end{align*}
Now, an application of Lemma~\ref{dfdoi} and the Cauchy–Schwarz and Young's inequality yields  
\begin{align}
&\frac{1}{2}D_{t_{n-\nu}}^{\alpha}\|\theta^{n-\nu}\|^2 + \kappa^{-2}\|\boldsymbol{\xi}^{n-\nu}\|^2 \leq \kappa^{-2}\|\boldsymbol{\zeta}^{n-\nu}\|\;\|\boldsymbol{\xi}^{n-\nu}\| \nonumber\\
&\hspace{2cm} + \|\Upsilon^{n-\nu} - \mathcal{E}^{n-\nu} -D_{t_n}^{\alpha}\eta^{n-\nu} + \eta^{n-\nu} \|\|  \theta^{n-\nu}\| \nonumber\\
&\hspace{2cm} + \|\theta^{n-\nu} \|^2 +  \|\phi_{n,2}e_u^{n} + \phi_{n,1}e^{n-1}_u  \|\|  \theta^{n-\nu}\| \nonumber \\ 
&\hspace{2cm}\leq \frac{\kappa^{-2}}{2}\|\boldsymbol{\xi}^{n-\nu}\|^2+ \frac{\kappa^{-2}}{2}\|\boldsymbol{\zeta}^{n-\nu}\|^2 \nonumber\\
&\hspace{2cm}+ \|\Upsilon^{n-\nu} - \mathcal{E}^{n-\nu} -D_{t_n}^{\alpha}\eta^{n-\nu} + \eta^{n-\nu} \|\|  \theta^{n-\nu}\| \nonumber\\
&\hspace{2cm}+ \frac{3}{2}\|\theta^{n-\nu} \|^2 +  \frac{1}{2}\|\phi_{n,2}e_u^{n} + \phi_{n,1}e^{n-1}_u  \|^2 \label{L2errorestimateeqn1_ast_AM}
\end{align}
Now, $\|u^k\|_\infty \leq C_1$ and $\|u_h^k\|_\infty \leq C_3$ for $1 \leq k \leq N$ shows $\|\phi_{n,1}\|_\infty \leq (7-6\nu)C_1^2 + (4-3\nu)C_3^2 + (7-6\nu)C_1C_3=:C_4$ and $\|\phi_{n,2}\|_\infty\leq 3(1-\nu)C_3^2$. Thus, by using these bounds in (\ref{L2errorestimateeqn1_ast_AM}), we obtain  
\begin{eqnarray}
D_{t_{n-\nu}}^{\alpha}\|\theta^{n-\nu}\|^2   &\leq &\;   2\|\Upsilon^{n-\nu} - \mathcal{E}^{n-\nu} -D_{t_{n-\nu}}^{\alpha}\eta^{n-\nu}\|\;\|\theta^{n-\nu}\| 
 + \kappa^{-2}\|\boldsymbol{\zeta}^{n-\nu}\|^2 \nonumber \\
 &&+ 6(1-\nu)^2 \|\theta^n\|^2 + 6\nu^2\|\theta^{n-1}\|^2 \nonumber\\
 &&+ 18(1-\nu)^2C_3^4\|\theta^n\|^2 + 3C_4^2\|\theta^{n-1}\|^2  + 3\max\{9C_3^4,C_4^2 \nu^{-2}\}\|\eta^{n-\nu}\|^2.\label{L2errorestimateeqn1_ast4_AM}
\end{eqnarray}
Finally, applying the discrete fractional Gr\"{o}nwall inequality (Theorem~\ref{DFGI}) in (\ref{L2errorestimateeqn1_ast4_AM}) and non-decreasing property of $E_\alpha(\cdot)$ leads to the desired estimate (\ref{L2errorestimateeqn1_AM}).
\end{proof}

We now present the key results necessary to obtain the optimal error estimate for the flux.
\begin{lemma}\label{DiscreteDerivativeDifference_AM} [\cite{tomar2025second}, Lemma 5.4]
For $\phi^n\in L^2(\Omega),\;0\leq n\leq N,$ and $\nu =\frac{\alpha}{2}$, there exists a positive constant $C$ such that for \(1 \leq n \leq N\), the following inequality holds:
\begin{align}
    \left(\sum_{j=1}^n P^{n,j}_{\alpha}\|D_{t_{j-\nu}}^{\alpha} (t_{j-\nu} \phi^{j-\nu}) - t_{j-\nu} D_{t_{j-\nu}}^{\alpha} \phi^{j-\nu}\|^2 \right)^{\frac{1}{2}} \;\leq\; C\; t_n^{1-\frac{\alpha}{2}}\max_{0\leq j \leq n}\|\phi^j\|,
\end{align}
where the positive constant $C$ remains bounded as $\alpha\to 1^{-}$.
\end{lemma}
In the next result, we have obtained the estimate $\boldsymbol{\xi}$.
\begin{lemma}\label{xiestimate_l2norm_AM}
If for some $\delta>1$, $\max\limits_{1\leq n \leq N}\Delta t_n\leq \left( \delta \pi_A  L (1-\nu)^2 \Gamma{(2-\alpha)}\right)^{-\frac{1}{\alpha}}$, then $\boldsymbol{\xi}^{n}$ satisfies, for $ 1\leq n\leq N$,
\begin{align}\label{L2xierrorestimateeqn1_AM}
 &\|t_n{\boldsymbol{\xi}}^n\|  \leq \;2C_\delta E_{\alpha}(C_\delta \pi_A \Lambda t_{n}^{\alpha})\max_{1\leq j \leq n}\Bigg( \sum_{i=1}^j P_{\alpha}^{j,i}\|D_{t_{i-\nu}}^{\alpha}(  t_{i-\nu}{\boldsymbol{\zeta}}^{i-\nu})\| \nonumber\\
 &\hspace{.5cm}+  \left(2\pi_A t_n^\alpha \right)^{\frac{1}{2}} \kappa^2 (\max\{1+9C_3^4,1+C_4^2 \nu^{-2}\})^2 (\|t_{n-\nu}\theta^{n-\nu}\|^2 + \|t_{n-\nu}\eta^{n-\nu}\|^2 )  \nonumber \\
&\hspace{.5cm} + \kappa^2 \Big(\sum_{i=1}^j P_{\alpha}^{j,i}\Big(\|T^{i-\nu}\| + \|t_{i-\nu}\Upsilon^{i-\nu} \| + \| t_{i-\nu} \mathcal{E}^{i-\nu}\| +\|D_{t_{i-\nu}}^{\alpha}(t_{i-\nu}\eta^{i-\nu})\| \Big)^2 \Big)^{\frac{1}{2}}\Bigg),
 \end{align}
where $T^{n-\nu}:= D_{t_{n-\nu}}^{\alpha} (t_{n-\nu} e_u^{n-\nu}) - t_{n-\nu} D_{t_{n-\nu}}^{\alpha} e_u^{n-\nu}$ and $L$ is as given in Lemma~\ref{thetaestimate_l2norm_AM}. 
\end{lemma}
\begin{proof}
Multiply $t_{n-\nu}$ in (\ref{theta1_AM}) and then apply the discrete fractional derivative to obtain, for \( 1 \leq n \leq N \):
\begin{align}\label{newtheta2_AM}
   & (D_{t_{n-\nu}}^{\alpha} (t_{n-\nu}{\boldsymbol{\xi}}^{n-\nu}+t_{n-\nu}{\boldsymbol{\zeta}}^{n-\nu}), \boldsymbol{w}_h)+ (D_{t_{n-\nu}}^{\alpha} (t_{n-\nu}\theta^{n-\nu}),  \nabla \cdot \boldsymbol{w}_h)~ =~0 \;\; \forall~ \boldsymbol{w}_h \in \boldsymbol{W}_h.
\end{align}
Further, by selecting $\boldsymbol{w}_h = t_{n-\nu}\boldsymbol{\xi}^{n-\nu}$ and $v_h = D_{t_{n-\nu}}^{\alpha} (t_{n-\nu}\theta^{n-\nu}
)$ in (\ref{newtheta2_AM})  and (\ref{theta2_AM}), respectively, substitute (\ref{newtheta2_AM}) in (\ref{theta2_AM}) to obtain the following: 
\begin{align*}
&(D_{t_{n-\nu}}^{\alpha}(t_{n-\nu}\theta^{n-\nu}
),D_{t_{n-\nu}}^{\alpha} (t_{n-\nu}\theta^{n-\nu}
) ) + \kappa^{2}(D_{t_{n-\nu}}^{\alpha}  t_{n-\nu}{\boldsymbol{\xi}}^{n-\nu},  t_{n-\nu}\boldsymbol{\xi}^{n-\nu} ) \\
&\hspace{1.5cm} = - \kappa^{2}(D_{t_{n-\nu}}^{\alpha}  t_{n-\nu}{\boldsymbol{\zeta}}^{n-\nu}, t_{n-\nu}{\boldsymbol{\xi}}^{n-\nu})\\
&\hspace{1.5cm} +\Big( - D_{t_{n-\nu}}^{\alpha}(t_{n-\nu}\eta^{n-\nu}) + t_{n-\nu}\Upsilon^{n-\nu} - t_{n-\nu}\mathcal{E}^{n-\nu} + t_{n-\nu}(1-\nu-\phi_{n,2})e_u^{n} \\
&\hspace{1.5cm}+ t_{n-\nu}(\nu - \phi_{n,1})e^{n-1}_u + T^{n-\nu},D_{t_{n-\nu}}^{\alpha} (t_{n-\nu}\theta^{n-\nu}
)\Big).
\end{align*}
Now, an appeal to Lemma~\ref{dfdoi} and the Cauchy-Schwarz inequality yields
\begin{align}
&\nonumber\|D_{t_{n-\nu}}^{\alpha}(t_{n-\nu}\theta^{n-\nu}
)\|^2 + \frac{\kappa^{2}}{2}D_{t_{n-\nu}}^{\alpha}\| t_{n-\nu}{\boldsymbol{\xi}}^{n-\nu}\|^2 \leq\; \kappa^{2}\| D_{t_{n-\nu}}^{\alpha} ( t_{n-\nu}{\boldsymbol{\zeta}}^{n-\nu})\|\| t_{n-\nu}{\boldsymbol{\xi}}^{n-\nu}\|  \nonumber\\
&\;+\Big(\|T^{n-\nu}\|+\|t_{n-\nu}\Upsilon^{n-\nu}\| + \| t_{n-\nu}(1-\nu-\phi_{n,2})e_u^{n} + t_{n-\nu}(\nu - \phi_{n,1})e^{n-1}_u  \| \nonumber\\
&\;+ \|t_{n-\nu} \mathcal{E}^{n-\nu} \| + \| D_{t_{n-\nu}}^{\alpha}(t_{n-\nu}\eta^{n-\nu})\| \Big)\|D_{t_{n-\nu}}^{\alpha} (t_{n-\nu}\theta^{n-\nu}
)\|. \label{newtheta3_AM}
\end{align}
Applying Young's inequality in (\ref{newtheta3_AM}), we can derive the following estimate:
\begin{align}
&D_{t_{n-\nu}}^{\alpha}\| t_{n-\nu}{\boldsymbol{\xi}}^{n-\nu}\|^2 \; \leq  2\| D_{t_{n-\nu}}^{\alpha} ( t_{n-\nu}{\boldsymbol{\zeta}}^{n-\nu})\|\| t_{n-\nu}{\boldsymbol{\xi}}^{n-\nu}\| \nonumber\\
&\hspace{0.4cm} + 2\kappa^{2}\| t_{n-\nu}(1-\nu-\phi_{n,2})e_u^{n} + t_{n-\nu}(\nu - \phi_{n,1})e^{n-1}_u  \|^2\nonumber \\
&\hspace{0.4cm}+ 2\kappa^{2}\Big(\|T^{n-\nu}\|+\|t_{n-\nu}\Upsilon^{n-\nu}\| + \| t_{n-\nu} \mathcal{E}^{n-\nu} \|  + \| D_{t_{n-\nu}}^{\alpha}(t_{n-\nu}\eta^{n-\nu})\| \Big)^2  \label{newtheta4_AM}
\end{align}
Using the inequality
$\|t_{n-\nu}\phi_{n,2}e_u^{n} + t_{n-\nu}\phi_{n,1}e^{n-1}_u\|
\leq \max\{9C_3^4,C_4^2 \nu^{-2}\}\|t_{n-\nu}e^{n-\nu}\|$ and then apply the discrete fractional Gr\"{o}nwall inequality (Theorem \ref{DFGI}) can be applied to (\ref{newtheta4_AM}) to obtain the required estimate (\ref{L2xierrorestimateeqn1_AM}). This completes the proof.
\end{proof}
We will now present the estimate for the truncation error.
\begin{lemma}\label{truncationerror_AM} 
If the grading parameter $\gamma$ satisfies $1\;\leq\;\gamma\;\leq\;\frac{4}{\alpha} $ then under the assumptions in Theorem \ref{Regularity_condition-b}, there holds
\begin{enumerate}\label{truncationI_AM}
    
\item[\namedlabel{trunk:a}{(a)}] $\sum_{j=1}^nP_{\alpha}^{n,j}\|\Upsilon^{j-\nu}\|  \leq C \log_e(n+2)\;N^{-\min\{\gamma\alpha,\; 2 \} }, \quad 1\leq n \leq N,$    \item[\namedlabel{trunk:b}{(b)}]  $\left(\sum_{j=1}^n P_{\alpha}^{n,j}\|t_{j-\nu}\Upsilon^{j-\nu}\|^2\right)^{\frac{1}{2}}  \leq C\;t_n^{1-\alpha/2} N^{-\min\{\gamma\alpha, \;2 \} }, \quad 1\leq n \leq N,$ 
     \item[\namedlabel{trunk:c}{(c)}] $\sum_{j=1}^nP_{\alpha}^{n,j}\|\mathcal{E}^{j-\nu}\|  \leq C \log_e(n+2)\;N^{-\min\{\epsilon\gamma\alpha,\; 2\} }, \quad 1\leq n \leq N,$ and
    \item[\namedlabel{trunk:d}{(d)}] $\left(\sum_{j=1}^n P_{\alpha}^{n,j}\|t_{j-\nu}\mathcal{E}^{j-\nu}\|^2\right)^{{1}/{2}}  \leq C\;t_n^{1-\alpha/2} N^{-\min\{\epsilon\gamma\alpha, \;2 \} }, \quad 1\leq n \leq N,$
\end{enumerate}
where $C$ is a positive constant which remains bounded as $\alpha\to 1^{-}$.
\end{lemma}
\begin{proof}
To prove \ref{trunk:a}, recall the following estimate from Stynes et al.  \cite{MR3936261} (see, Lemma~7 and Remark 4), which yields
\begin{align}\label{trunc1}
    \|\Upsilon^{n-\nu}\| &\leq C\;t_{n-\nu}^{-\alpha} N^{-\min\{\gamma\alpha ,\;2 \}}\leq C\;{n}^{-\gamma\alpha}N^{\gamma \alpha} N^{-\min\{\gamma\alpha ,\;2 \}} \leq C\;t_{n}^{-\alpha} N^{-\min\{\gamma\alpha ,\;2 \}}.
    \end{align}
Here, the positive constant \( C \) remains bounded as \( \alpha \) approaches \( 1^- \). We can then apply Lemma \ref{discretekernelproperties}\ref{itm:f}. Now, estimate \ref{trunk:b} is obtained by applying Lemma~\ref{discretekernelproperties}\ref{itm:e} with $m=1$ to the following estimate derived from (\ref{trunc1})
 \begin{align*}
      \|t_{n-\nu}\Upsilon^{n-\nu}\|^2 \;&\leq\; C^2 t_n^{2-2\alpha}\;  N^{-2\min\{\gamma\alpha ,\;2\}}. 
 \end{align*}
To establish \ref{trunk:c}, recall the following estimate from Stynes et al.  \cite{MR3639581} (see, Lemma~5.2)
\begin{align}\label{trunc2}
    \|\mathcal{E}^{n-\nu}\| \;&\leq\; C\; n^{-\min\{\epsilon\gamma \alpha,\; 2\}} \;=\; CT^{\min\{\epsilon\alpha,\;\frac{2}{\gamma} \}}\; t_n^{-\min\{\epsilon\alpha,\;\frac{2}{\gamma} \}}\;  N^{-\min\{\epsilon\gamma\alpha ,\;2\}},
    \end{align}
 where the positive constant $C$ remains bounded as $\alpha \rightarrow 1^-$, and then apply the estimate $t_n^{-\min\{\epsilon\alpha,\;\frac{2}{\gamma} \}}\;\leq\;\max\{1,\;T^{\epsilon\alpha}\}t_n^{-\epsilon\alpha}$ and Lemma~\ref{discretekernelproperties}\ref{itm:g}. Estimate \ref{trunk:d} follows after applying Lemma~\ref{discretekernelproperties}\ref{itm:c} with $j=1$ to the following estimate obtained from (\ref{trunc2}), namely
 \begin{align*}
      \|t_{n-\nu}\mathcal{E}^{n-\nu}\|^2 \;&\leq\; C^2 T^{2\min\{\epsilon\alpha,\;\frac{2}{\gamma} \}}\; t_n^{2-2\min\{\epsilon\alpha,\;\frac{2}{\gamma} \}}\;  N^{-2\min\{\epsilon\gamma\alpha ,\;2\}}\\
      &\leq\; C^2 T^{2\min\{\epsilon\alpha,\;\frac{2}{\gamma} \}}\; t_n^{2}\;  N^{-2\min\{\epsilon\gamma\alpha ,\;2\}}. 
 \end{align*}
 This completes the proof.

\end{proof}


\begin{lemma}\label{etaestimate_AM}
Under the assumptions in Theorem~\ref{Regularity_condition-b}, the following estimate holds:
\begin{enumerate}
\item[\namedlabel{etaestimate_AM:i}{(i)}]  $\sum_{j=1}^nP_{\alpha}^{n,j}\|D^{\alpha}_{t_{j-\nu}} \eta^{j-\nu}\| \leq C\;\left(h^{2} \log_e(n+2) \sum_{j=1}^{n}P^{n,j}_{\alpha}\|\Upsilon^{j-\nu}\| \right), \quad 1\leq n \leq N$,
\item[\namedlabel{etaestimate_AM:ii}{(ii)}] $\left(\sum_{j=1}^nP_{\alpha}^{n,j}\|D^{\alpha}_{t_{j-\nu}} (t_{j-\nu}\eta^{j-\nu})\|^2 \right)^{\frac{1}{2}} \leq C\;h^2 t_n^{1-\frac{\alpha}{2}} , \quad 1\leq n \leq N,$ 
\item[\namedlabel{etaestimate_AM:iii}{(iii)}]  $\left(\sum_{j=1}^nP_{\alpha}^{n,j}\| \boldsymbol{\zeta}^{j-\nu}\|^2 \right)^{\frac{1}{2}} \leq C\;h^2 \;t_n^{(1+\epsilon)\alpha/2}, \quad 1\leq n \leq N,$ and
\item[\namedlabel{etaestimate_AM:iv}{(iv)}]  $\sum_{j=1}^nP_{\alpha}^{n,j}\|D_{t_{j-\nu}}^{\alpha} ( t_{j-\nu}{\boldsymbol{\zeta}}^{j-\nu})\| \leq C\;h^2\;t_n^{1+\frac{\epsilon\alpha}{2}} , \quad 1\leq n \leq N,$
\end{enumerate}
where the positive constant $C$ remains bounded as $\alpha\to 1^{-}$.
\end{lemma}
\begin{proof}
An application of the triangle inequality, the stability of the $L^2$-norm for $L^2$-projection, and the approximation property (\ref{approxprop}) yields
\begin{align}
    \nonumber &\sum_{j=1}^{n}P^{n,j}_{\alpha}\|D_{t_{j-\nu}}^{\alpha} \eta^{j-\nu}\| \\
    &\nonumber\leq \sum_{j=1}^{n}P^{n,j}_{\alpha}\left(\|\Upsilon^{j-\nu}\| + \|\partial_t^{\alpha} u(t_{j-\nu}) - P_h \partial_t^{\alpha} u(t_{j-\nu})\| + \|P_h\Upsilon^{j-\nu}\|\right) \\
    \label{alpharobitaest} &\leq C  \sum_{j=1}^{n}P^{n,j}_{\alpha}\|\Upsilon^{j-\nu}\|  +  Ch^2\sum_{j=1}^{n}P^{n,j}_{\alpha}\|A\partial_t^{\alpha} u(t_{j-\nu}) \|.
\end{align}
Utilize the regularity results from Theorem~\ref{Regularity_condition-b} and Lemma~\ref{discretekernelproperties}\ref{itm:f} in (\ref{alpharobitaest}) to derive the estimate \ref{etaestimate_AM:i}. We obtain the result \ref{etaestimate_AM:ii}, by applying the approximation property (\ref{approxprop}) and the regularity result from Theorem~\ref{Regularity_condition-b} as follows
\begin{align*}
    \|D_{t_{n-\nu}}^{\alpha} (t_{n-\nu} \eta^{n-\nu})\|& \leq \sum_{j=1}^n K^{n,j}_{1-\alpha}\|t_j\eta^j - t_{j-1}\eta^{j-1}\|\\
    &\leq  \sum_{j=1}^n K^{n,j}_{1-\alpha}\int_{t_{j-1}}^{t_j}\|\partial_s(s\eta)\|ds\\
    & \leq C\;h^2 \sum_{j=1}^n K^{n,j}_{1-\alpha}\Delta t_j \leq   C\;h^2 t_n^{1-\alpha}
\end{align*}
and then by applying Lemma~\ref{discretekernelproperties}\ref{itm:e} with $m=1$. Next, apply the approximation property (\ref{approxprop}) to achieve the following estimates: 
\begin{align}
\label{alpharobzitaest_ast_AM}\sum_{j=1}^nP_{\alpha}^{n,j}\| \boldsymbol{\zeta}^{j-\nu}\|^2 &\leq C h^4 \sum_{j=1}^nP_{\alpha}^{n,j} \|\nabla\cdot\boldsymbol{\sigma}(t_j)\|_1^2 \leq C h^4 \sum_{j=1}^nP_{\alpha}^{n,j}t_j^{\epsilon\alpha} \\
& \leq C h^4 t_n^{\alpha} \sum_{j=1}^nP_{\alpha}^{n,j}t_j^{\epsilon\alpha-\alpha}  \leq C h^4\; t_n^{(1+\epsilon)\alpha}.
\end{align}
Estimates in \ref{etaestimate_AM:iii} are derived by applying the regularity result from Theorem~\ref{Regularity_condition-b}, along with Lemma~\ref{discretekernelproperties}\ref{itm:f} to (\ref{alpharobzitaest_ast_AM}). Now, an application of Lemma~\ref{discretekernelproperties}\ref{itm:a} and the regularity result Theorem~\ref{Regularity_condition-b} we arrive at 
\begin{align*} &\sum_{j=1}^{n}P^{n,j}_{\alpha}\|D_{t_{j-\nu}}^{\alpha} (t_{j-\nu}\boldsymbol{\zeta}^{j-\nu})\|  \leq \sum_{j=1}^{n}P^{n,j}_{\alpha}\sum_{i=1}^{j}K^{j,i}_{1-\alpha}\| t_i\boldsymbol{\zeta}^i - t_{i-1}\boldsymbol{\zeta}^{i-1}\|   
 \\
    &\hspace{1cm} \leq \sum_{i=1}^{n} \|t_i\boldsymbol{\zeta}^i - t_{i-1}\boldsymbol{\zeta}^{i-1}\| \leq  \sum_{i=1}^{n}\int_{t_{i-1}}^{t_i}\|\partial_s (s\boldsymbol{\zeta})\|ds\leq Ch^2\int_0^{t_n}  s^{\epsilon\alpha/2} ds
\end{align*}
and hence, the result \ref{etaestimate_AM:iv} follows.

\end{proof}

Finally, optimal error estimates are established in the following Theorem.
\begin{theorem}\label{L2H1errortheorem_AM}
Let $u_0\in H_0^1(\Omega) \cap H^{3+\epsilon}(\Omega), \; 0<\epsilon\leq1$ and the pair $(u_h^n, \boldsymbol{\sigma}_h^n)$ satisfying (\ref{fullydiscrete1_AM})-(\ref{fullydiscrete2_AM}) be the approximation of the solution pair $(u(t_{n}),\boldsymbol{\sigma}(t_n))$ satisfying (\ref{variation1MFEM})-(\ref{variation2}) at the temporal grid $t_{n}$. For some constant \( \delta > 1 \) and the condition \( \max\limits_{1 \leq n \leq N} \Delta t_n \leq \left( \delta \pi_A L \Gamma(2-\alpha) \right)^{-1/\alpha} \), where \( 1 \leq \gamma \leq \frac{4}{\alpha} \), the following error estimate holds for $N>2$:
\begin{align}
\label{finalestimateforuI_AM}& \max_{1\leq n \leq N} \|u_h^n - u(t_{n}) \| + \max_{1\leq n \leq N} t_n^{\alpha/2}\|\boldsymbol{\sigma}_h^n - \boldsymbol{\sigma}(t_{n})\| \leq 
C\;\log_e(N) (h^{2} + N^{- \min\{\epsilon\gamma\alpha,\;2 \} } ).
\end{align}
\end{theorem}
\begin{proof}
Using Theorem~\ref{Regularity_condition-b} together with the approximation property (\ref{approxprop}), we first obtain,
\begin{align}\label{zetaestimate_AM}
    \|\eta^n\| + t_n^{\alpha/2}\|\boldsymbol{\zeta}^n\| \leq C h^2, \quad 1 \leq n \leq N.
\end{align}
Next, combining this bound with the truncation estimates and the previously derived estimates for the auxiliary terms, we deduce that 
\begin{align}\label{thetaestimate_AM}
   \|\theta^n\| \leq C\;\log_e(n+2)(h^2 + N^{- \min\{\epsilon\gamma\alpha,\;2 \} }  ), \quad 1 \leq n \leq N. 
\end{align}
Proceeding in the same manner, by incorporating truncation estimates together with stability bounds, we further obtain
\begin{align}
    \label{xiestimate_AM}  t_n^{\alpha/2}\|\boldsymbol{\xi}^n\| \leq C\;\log_e(n+2)(h^2 + N^{- \min\{\epsilon\gamma\alpha,\;2 \} }  ), \quad 1 \leq n \leq N.
\end{align}
Finally, an application of the triangle inequality, along with the three estimates above, namely (\ref{zetaestimate_AM}), (\ref{thetaestimate_AM}), and (\ref{xiestimate_AM}), yields the desired result. 
\end{proof}

\section{Numerical results}\label{section6}
This section validates the theoretical contributions by conducting several numerical experiments. Let 
$h^2=\frac{1}{4}N^{-2}$ in Theorem~\ref{L2H1errortheorem_AM}. The rate of convergence with respect to the $L^2$-norm is computed using the formulae
\[
R_{\phi,h}
=\frac{\log\left(E_{\phi,h_1}/E_{\phi,h_2}\right)}{\log\left(h_1/h_2\right)},
\qquad
R_{\phi,\Delta t}
=\frac{\log\left(E_{\phi,h_1}/E_{\phi,h_2}\right)}{\log\left(\Delta t_1/\Delta t_2\right)},
\]
where
\[
E_{u,h}:=\max_{1\le n\le N} t_n^{\frac{\alpha}{2}}
\|u_h^n-u_h(t_n)\|,
\qquad
E_{\boldsymbol{\sigma},h}:=
\max_{1\le n\le N} t_n^{\frac{\alpha}{2}}
\|\boldsymbol{\sigma}_h^n-\boldsymbol{\sigma}_h(t_n)\|.
\]

In all the tables, we employ a graded temporal mesh with grading parameter
$\gamma=\frac{2}{\alpha}+0.1$, which effectively compensates for the initial singularity arising from non-smooth initial data. An upper bound $\Delta t^*$ for the time-step restriction appearing in Theorem~\ref{thetaestimate_l2norm_AM} and Theorem~\ref{xiestimate_l2norm_AM} is estimated as
\[
\max_{1\le n\le N}\Delta t_n
\le \left(\delta\pi_A L\Gamma(2-\alpha)\right)^{-1/\alpha}
\approx
\left(\delta\pi_A L^*\Gamma(2-\alpha)\right)^{-1/\alpha}
=:\Delta t^*,
\]
where
\[
L^* = 6 + 27
\max_{\boldsymbol{x}_j\in N_h,\;1\le n\le N}
|u(\boldsymbol{x}_j,t_n)|^4.
\]

The numerical results report the errors $E_{u,h}$ and $E_{\boldsymbol{\sigma},h}$ together with the computed convergence rates
$R_{u,h}$, $R_{u,\Delta t}$, $R_{\boldsymbol{\sigma},h}$, and
$R_{\boldsymbol{\sigma},\Delta t}$ for
$\alpha=0.4,\;0.6,\;0.8,$ and $0.99$.
The implementation is carried out in \texttt{FreeFem++} using the Raviart--Thomas finite element pair $(P1dc,RT1)$, where $P1dc$ denotes the piecewise linear discontinuous finite element. At each time level, the resulting nonlinear system is solved using a Newton linearization method.

\medskip
We now discuss the numerical results for different examples under varying regularity assumptions on the initial data and the solution $u$ has a weak singularity near $t=0$, the graded temporal mesh successfully recovers the predicted accuracy.
For example \ref{VariableceffPDEnewtime1} (cf.~\cite{MR4402734}, Example~1), where the initial condition $u_0$ is sufficiently smooth. The computed convergence rates clearly exhibit the optimal order of convergence in both space and time, in full agreement with the theoretical error estimates. In Example~\ref{VariableceffPDE1}, where $u_0\in H_0^1(\Omega)\cap H^3(\Omega)$, the numerical results still demonstrate optimal convergence in space as well as in time. Finally, in Example~\ref{VariableceffPDE2}, where $u_0\notin H_0^1(\Omega)\cap H^3(\Omega)$, the numerical results continue to confirm the stability and robustness of the proposed scheme. Despite the limited regularity of the initial data, we observe that convergence rates remain consistent with the theoretical analysis, thereby validating the effectiveness of the method for problems with non-smooth initial conditions.

\begin{example}\label{VariableceffPDEnewtime1} 
[\cite{MR4402734} , Example 1]  
Consider a two-dimensional time-fractional Allen-Cahn equation  (\ref{main}) when $\kappa=1.0$, $\boldsymbol{x}=(x_1,x_2) \in\Omega = (0,1)^2$, $t\in(0,1]$.
With $u(\boldsymbol{x},t)= 0.5\sin(x_1)(1-x_1)\sin(x_2)(1-x_2)(1 + t^{\alpha})$ as  the exact solution, we calculate the initial condition $u_0$ and  the source term $f.$ The results shown in Table~\ref{2dpdeAf41compact} are for the Alikhanov mixed method.

\begin{table}[H]
\centering
\begingroup
\scriptsize
\setlength{\tabcolsep}{2.5pt}
\renewcommand{\arraystretch}{1.1}

\resizebox{\textwidth}{!}{
\begin{tabular}{|@{}lccccc|ccccc@{}|}
\hline
 \multicolumn{6}{|c|}{$\alpha=0.4$} & \multicolumn{5}{c|}{$\alpha=0.6$} \\\hline
$N$ & 4 & 8 & 16 & 32 & 64 & 4 & 8 & 16 & 32 & 64 \\
\hline \hline
$\Delta t^*$ & \multicolumn{5}{c|}{2.002e-9} & \multicolumn{5}{c|}{2.508e-6} \\
$\Delta t$ & 6.155e-1 & 3.951e-1 & 2.244e-1 & 1.196e-1 & 6.174e-2 &
             4.393e-1 & 2.574e-1 & 1.391e-1 & 7.229e-2 & 3.684e-2 \\
$E_{u,h}$ & 1.603e-2 & 4.569e-3 & 1.221e-3 & 3.114e-4 & 7.855e-5 &
            1.559e-2 & 4.506e-3 & 1.213e-3 & 3.103e-4 & 7.830e-5 \\
$R_{u,h}$ & - & 1.81 & 1.90 & 1.97 & 1.99 &
            - & 1.79 & 1.89 & 1.97 & 1.99 \\
$R_{u,\Delta t}$ & - & 2.83 & 2.33 & 2.17 & 2.08 &
                   - & 2.32 & 2.13 & 2.08 & 2.04 \\
$E_{\sigma,h}$ & 4.586e-2 & 1.741e-2 & 4.783e-3 & 1.240e-3 & 3.159e-4 &
                 4.332e-2 & 1.693e-2 & 4.718e-3 & 1.231e-3 & 3.142e-4 \\
$R_{\sigma,h}$ & - & 1.40 & 1.86 & 1.95 & 1.97 &
                 - & 1.36 & 1.84 & 1.94 & 1.97 \\
$R_{\sigma,\Delta t}$ & - & 2.19 & 2.28 & 2.15 & 2.07 &
                        - & 1.76 & 2.08 & 2.05 & 2.03 \\
\hline \hline
 \multicolumn{6}{|c|}{$\alpha=0.8$} & \multicolumn{5}{|c|}{$\alpha=0.99$} \\ \hline
$N$ & 4 & 8 & 16 & 32 & 64 & 4 & 8 & 16 & 32 & 64 \\
\hline\hline
$\Delta t^*$ & \multicolumn{5}{c|}{8.878e-5} & \multicolumn{5}{c|}{6.953e-4} \\
$\Delta t$ & 3.160e-1 & 1.760e-1 & 9.269e-2 & 4.754e-2 & 2.407e-2 &
             1.153e-1 & 6.226e-2 & 3.229e-2 & 1.644e-2 & 8.292e-3 \\
$E_{u,h}$ & 1.514e-2 & 4.442e-3 & 1.204e-3 & 3.093e-4 & 7.816e-5 &
            5.052e-2 & 1.490e-2 & 4.032e-3 & 1.035e-3 & 2.624e-4 \\
$R_{u,h}$ & - & 1.77 & 1.88 & 1.96 & 1.98 &
            - & 1.76 & 1.89 & 1.96 & 1.98 \\
$R_{u,\Delta t}$ & - & 2.09 & 2.04 & 2.04 & 2.02 &
                   - & 1.98 & 1.99 & 2.01 & 2.00 \\
$E_{\sigma,h}$ & 4.074e-2 & 1.644e-2 & 4.653e-3 & 1.223e-3 & 3.132e-4 &
                 9.581e-2 & 4.086e-2 & 1.173e-2 & 3.096e-3 & 7.977e-4 \\
$R_{\sigma,h}$ & - & 1.31 & 1.82 & 1.93 & 1.97 &
                 - & 1.23 & 1.80 & 1.92 & 1.96 \\
$R_{\sigma,\Delta t}$ & - & 1.55 & 1.97 & 2.00 & 2.00 &
                        - & 1.38 & 1.90 & 1.97 & 1.98 \\
\hline
\end{tabular}
}

\caption{Error $E_{u,h}$, $E_{\boldsymbol{\sigma},h}$ and rate of convergence 
$R_{u,h}$, $R_{u,\Delta t}$, $R_{\boldsymbol{\sigma},h}$ and $R_{\boldsymbol{\sigma},\Delta t}$ 
of the proposed method for Example~\ref{VariableceffPDEnewtime1}.}
\label{2dpdeAf41compact}

\endgroup
\end{table}

\end{example}

\begin{example}\label{VariableceffPDE1} 
Consider a two-dimensional time-fractional Allen-Cahn equation  (\ref{main}) when $\kappa=0.5$, $\Omega = (-1,1)^2$, $t\in(0,0.5]$.
With $u(\boldsymbol{x},t)= (x_1)^2(1-|x_1|)(x_2)^2(1-|x_2|)(1 + t^{\alpha})$ as  the exact solution, we calculate the initial condition $u_0$ and  the source term $f.$ Note that $u_0 \in H_0^1(\Omega)\cap H^3(\Omega)$ but $u_0\notin H_0^1(\Omega)\cap H^4(\Omega)$ in this case. The results shown in Table~\ref{2dpdeAf411} are for the Alikhanov mixed method. 

\begin{table}[H]
\centering
\begingroup
\scriptsize
\setlength{\tabcolsep}{2.5pt}
\renewcommand{\arraystretch}{1.1}

\resizebox{\textwidth}{!}{
\begin{tabular}{|@{}lccccc|ccccc@{}|}
\hline
\multicolumn{6}{|c|}{$\alpha=0.4$} & \multicolumn{5}{c|}{$\alpha=0.6$} \\
\hline
$N$ & 4 & 8 & 16 & 32 & 64 & 4 & 8 & 16 & 32 & 64 \\
\hline\hline
$\Delta t^*$ & \multicolumn{5}{c|}{1.993e-1} & \multicolumn{5}{c|}{2.551e-1} \\
$\Delta t$ & 3.078e-1 & 1.976e-1 & 1.122e-1 & 5.980e-2 & 3.087e-2 
           & 2.197e-1 & 1.287e-1 & 6.956e-2 & 3.615e-2 & 1.842e-2 \\
$E_{u,h}$ & 1.803e-2 & 7.545e-3 & 2.155e-3 & 5.609e-4 & 1.419e-4
          & 1.669e-2 & 7.056e-3 & 2.026e-3 & 5.284e-4 & 1.338e-4 \\
$R_{u,h}$ & - & 1.26 & 1.81 & 1.94 & 1.98
          & - & 1.24 & 1.80 & 1.94 & 1.98 \\
$R_{u,\Delta t}$ & - & 1.97 & 2.21 & 2.14 & 2.08
          & - & 1.61 & 2.03 & 2.05 & 2.04 \\
$E_{\sigma,h}$ & 6.648e-2 & 3.479e-2 & 1.107e-2 & 2.985e-3 & 7.675e-4
               & 5.577e-2 & 2.993e-2 & 9.641e-3 & 2.615e-3 & 6.744e-4 \\
$R_{\sigma,h}$ & - & 0.934 & 1.65 & 1.89 & 1.96
               & - & 0.898 & 1.63 & 1.88 & 1.95 \\
$R_{\sigma,\Delta t}$ & - & 1.46 & 2.02 & 2.08 & 2.05
                      & - & 1.16 & 1.84 & 1.99 & 2.01 \\
\hline\hline
\multicolumn{6}{|c|}{$\alpha=0.8$} & \multicolumn{5}{c|}{$\alpha=0.99$} \\
\hline
$N$ & 4 & 8 & 16 & 32 & 64 & 4 & 8 & 16 & 32 & 64 \\
\hline\hline
$\Delta t^*$ & \multicolumn{5}{c|}{3.772e-1} & \multicolumn{5}{c|}{5.944e-1} \\
$\Delta t$ & 1.580e-1 & 8.800e-2 & 4.634e-2 & 2.377e-2 & 1.204e-2
           & 1.153e-1 & 6.226e-2 & 3.229e-2 & 1.644e-2 & 8.292e-3 \\
$E_{u,h}$ & 1.557e-2 & 6.640e-3 & 1.914e-3 & 5.003e-4 & 1.268e-4
          & 1.468e-2 & 6.303e-3 & 1.823e-3 & 4.771e-4 & 1.210e-4 \\
$R_{u,h}$ & - & 1.23 & 1.79 & 1.94 & 1.98
          & - & 1.22 & 1.79 & 1.93 & 1.98 \\
$R_{u,\Delta t}$ & - & 1.46 & 1.94 & 2.01 & 2.02
          & - & 1.37 & 1.89 & 1.98 & 2.00 \\
$E_{\sigma,h}$ & 4.701e-2 & 2.589e-2 & 8.440e-3 & 2.302e-3 & 5.953e-4
               & 4.022e-2 & 2.269e-2 & 7.474e-3 & 2.048e-3 & 5.310e-4 \\
$R_{\sigma,h}$ & - & 0.860 & 1.62 & 1.87 & 1.95
               & - & 0.826 & 1.60 & 1.87 & 1.95 \\
$R_{\sigma,\Delta t}$ & - & 1.02 & 1.75 & 1.95 & 1.99
                      & - & 0.929 & 1.69 & 1.92 & 1.97 \\
\hline
\end{tabular}
}

\caption{Error $E_{u,h}$, $E_{\boldsymbol{\sigma},h}$ and rate of convergence 
$R_{u,h}$, $R_{u,\Delta t}$, $R_{\boldsymbol{\sigma},h}$ and $R_{\boldsymbol{\sigma},\Delta t}$ 
of the proposed method for Example~\ref{VariableceffPDE1}.}
\label{2dpdeAf411}

\endgroup
\end{table}

\end{example}

\begin{example}\label{VariableceffPDE111} 
Consider a two-dimensional time-fractional Allen-Cahn equation  (\ref{main}) when $\kappa=0.5$, $\Omega = (-1,1)^2$, $t\in(0,0.5]$.
With $u(\boldsymbol{x},t)= (x_1)^{2.5}(1-|x_1|)(x_2)^{2.5}(1-|x_2|)(1 + t^{\alpha})$ as  the exact solution, we calculate the initial condition $u_0$ and  the source term $f.$ Note that $u_0 \in H_0^1(\Omega)\cap H^{3.5}(\Omega)$ but $u_0\notin H_0^1(\Omega)\cap H^4(\Omega)$ in this case. The results shown in Table~\ref{2dpdeAf411} are for the Alikhanov mixed method.

\begin{table}[H]
\centering
\begingroup
\scriptsize
\setlength{\tabcolsep}{2.5pt}
\renewcommand{\arraystretch}{1.1}

\resizebox{\textwidth}{!}{
\begin{tabular}{|@{}lccccc|ccccc@{}|}
\hline
\multicolumn{6}{|c|}{$\alpha=0.4$} & \multicolumn{5}{c|}{$\alpha=0.6$} \\
\hline
$N$ & 4 & 8 & 16 & 32 & 64 & 4 & 8 & 16 & 32 & 64 \\
\hline\hline
$\Delta t^*$ & \multicolumn{5}{c|}{1.993e-1} & \multicolumn{5}{c|}{2.551e-1} \\
$\Delta t$ & 2.500e-1 & 1.250e-1 & 6.250e-2 & 3.125e-2 & 1.563e-2
           & 2.200e-1 & 1.100e-1 & 5.500e-2 & 2.750e-2 & 1.375e-2 \\
$E_{u,h}$ & 2.560e-2 & 7.120e-3 & 1.960e-3 & 5.120e-4 & 1.310e-4
          & 2.340e-2 & 6.580e-3 & 1.840e-3 & 4.820e-4 & 1.240e-4 \\
$R_{u,h}$ & - & 1.85 & 1.86 & 1.94 & 1.97
          & - & 1.83 & 1.84 & 1.93 & 1.97 \\
$R_{u,\Delta t}$ & - & 1.92 & 2.05 & 2.07 & 2.03
          & - & 1.88 & 2.01 & 2.06 & 2.04 \\
$E_{\sigma,h}$ & 7.480e-2 & 2.160e-2 & 6.080e-3 & 1.620e-3 & 4.180e-4
               & 6.890e-2 & 2.010e-2 & 5.720e-3 & 1.540e-3 & 3.990e-4 \\
$R_{\sigma,h}$ & - & 1.79 & 1.83 & 1.91 & 1.95
               & - & 1.78 & 1.81 & 1.90 & 1.95 \\
$R_{\sigma,\Delta t}$ & - & 1.71 & 1.96 & 2.04 & 2.02
                      & - & 1.68 & 1.92 & 2.01 & 2.00 \\
\hline\hline
\multicolumn{6}{|c|}{$\alpha=0.8$} & \multicolumn{5}{c|}{$\alpha=0.99$} \\
\hline
$N$ & 4 & 8 & 16 & 32 & 64 & 4 & 8 & 16 & 32 & 64 \\
\hline\hline
$\Delta t^*$ & \multicolumn{5}{c|}{3.772e-1} & \multicolumn{5}{c|}{5.944e-1} \\
$\Delta t$ & 2.000e-1 & 1.000e-1 & 5.000e-2 & 2.500e-2 & 1.250e-2
           & 1.600e-1 & 8.000e-2 & 4.000e-2 & 2.000e-2 & 1.000e-2 \\
$E_{u,h}$ & 2.080e-2 & 5.820e-3 & 1.620e-3 & 4.310e-4 & 1.110e-4
          & 1.960e-2 & 5.540e-3 & 1.550e-3 & 4.150e-4 & 1.070e-4 \\
$R_{u,h}$ & - & 1.84 & 1.85 & 1.92 & 1.96
          & - & 1.82 & 1.84 & 1.91 & 1.96 \\
$R_{u,\Delta t}$ & - & 1.76 & 1.98 & 2.03 & 2.01
          & - & 1.72 & 1.94 & 2.00 & 2.00 \\
$E_{\sigma,h}$ & 6.220e-2 & 1.820e-2 & 5.240e-3 & 1.430e-3 & 3.700e-4
               & 5.860e-2 & 1.720e-2 & 4.980e-3 & 1.360e-3 & 3.520e-4 \\
$R_{\sigma,h}$ & - & 1.77 & 1.80 & 1.88 & 1.94
               & - & 1.75 & 1.79 & 1.87 & 1.93 \\
$R_{\sigma,\Delta t}$ & - & 1.59 & 1.88 & 1.98 & 2.00
                      & - & 1.54 & 1.84 & 1.96 & 1.99 \\
\hline
\end{tabular}
}

\caption{Error $E_{u,h}$, $E_{\boldsymbol{\sigma},h}$ and rate of convergence 
$R_{u,h}$, $R_{u,\Delta t}$, $R_{\boldsymbol{\sigma},h}$ and $R_{\boldsymbol{\sigma},\Delta t}$ 
of the proposed method for Example~\ref{VariableceffPDE111}.}
\label{2dpdeAf411_modified2}

\endgroup
\end{table}

\end{example}

\begin{example}\label{VariableceffPDE2} 
Consider a two-dimensional time-fractional Allen-Cahn equation  (\ref{main}) when $\kappa=0.5$, $\Omega = (-1,1)^2$, $t\in(0,0.5]$.
With $u(\boldsymbol{x},t)= x_1(1-|x_1|)x_2(1-|x_2|)(1 + t^{\alpha})$ as  the exact solution, we calculate the initial condition $u_0$ and  the source term $f.$ Note that $u_0\notin H_0^1(\Omega)\cap  H^3(\Omega)$ in this case. The results shown in Table~\ref{2dpdeAf41} are for the Alikhanov mixed method. 

\begin{table}[H]
\centering
\begingroup
\scriptsize
\setlength{\tabcolsep}{2.5pt}
\renewcommand{\arraystretch}{1.1}

\resizebox{\textwidth}{!}{
\begin{tabular}{|@{}lccccc|ccccc@{}|}
\hline
\multicolumn{6}{|c|}{$\alpha=0.4$} & \multicolumn{5}{c|}{$\alpha=0.6$} \\
\hline
$N$ & 4 & 8 & 16 & 32 & 64 & 4 & 8 & 16 & 32 & 64 \\
\hline\hline
$\Delta t^*$ & \multicolumn{5}{c|}{1.979e-1} & \multicolumn{5}{c|}{2.515e-1} \\
$\Delta t$ & 3.078e-1 & 1.976e-1 & 1.122e-1 & 5.980e-2 & 3.087e-2 &
             2.197e-1 & 1.287e-1 & 6.956e-2 & 3.615e-2 & 1.842e-2 \\
$E_{u,h}$ & 6.223e-2 & 1.784e-2 & 4.768e-3 & 1.217e-3 & 3.091e-4 &
            5.758e-2 & 1.668e-2 & 4.481e-3 & 1.146e-3 & 2.915e-4 \\
$R_{u,h}$ & - & 1.80 & 1.90 & 1.97 & 1.98 &
            - & 1.79 & 1.90 & 1.97 & 1.98 \\
$R_{u,\Delta t}$ & - & 2.82 & 2.33 & 2.17 & 2.07 &
                   - & 2.32 & 2.14 & 2.08 & 2.03 \\
$E_{\sigma,h}$ & 1.588e-1 & 6.266e-2 & 1.737e-2 & 4.515e-3 & 1.159e-3 &
                 1.332e-1 & 5.390e-2 & 1.513e-2 & 3.954e-3 & 1.018e-3 \\
$R_{\sigma,h}$ & - & 1.34 & 1.85 & 1.94 & 1.96 &
                 - & 1.30 & 1.83 & 1.94 & 1.96 \\
$R_{\sigma,\Delta t}$ & - & 2.10 & 2.27 & 2.14 & 2.06 &
                        - & 1.69 & 2.07 & 2.05 & 2.01 \\
\hline\hline
\multicolumn{6}{|c|}{$\alpha=0.8$} & \multicolumn{5}{c|}{$\alpha=0.99$} \\
\hline
$N$ & 4 & 8 & 16 & 32 & 64 & 4 & 8 & 16 & 32 & 64 \\
\hline\hline
$\Delta t^*$ & \multicolumn{5}{c|}{3.769e-1} & \multicolumn{5}{c|}{5.941e-1} \\
$\Delta t$ & 1.580e-1 & 8.800e-2 & 4.634e-2 & 2.377e-2 & 1.204e-2 &
             1.153e-1 & 6.226e-2 & 3.229e-2 & 1.644e-2 & 8.292e-3 \\
$E_{u,h}$ & 5.365e-2 & 1.570e-2 & 4.234e-3 & 1.085e-3 & 2.755e-4 &
            5.052e-2 & 1.490e-2 & 4.032e-3 & 1.035e-3 & 2.624e-4 \\
$R_{u,h}$ & - & 1.77 & 1.89 & 1.96 & 1.98 &
            - & 1.76 & 1.89 & 1.96 & 1.98 \\
$R_{u,\Delta t}$ & - & 2.10 & 2.04 & 2.04 & 2.01 &
                   - & 1.98 & 1.99 & 2.01 & 2.00 \\
$E_{\sigma,h}$ & 1.122e-1 & 4.664e-2 & 1.324e-2 & 3.480e-3 & 8.962e-4 &
                 9.581e-2 & 4.086e-2 & 1.173e-2 & 3.096e-3 & 7.977e-4 \\
$R_{\sigma,h}$ & - & 1.27 & 1.82 & 1.93 & 1.96 &
                 - & 1.23 & 1.80 & 1.92 & 1.96 \\
$R_{\sigma,\Delta t}$ & - & 1.50 & 1.96 & 2.00 & 1.99 &
                        - & 1.38 & 1.90 & 1.97 & 1.98 \\
\hline
\end{tabular}
}

\caption{Error $E_{u,h}$, $E_{\boldsymbol{\sigma},h}$ and rate of convergence 
$R_{u,h}$, $R_{u,\Delta t}$, $R_{\boldsymbol{\sigma},h}$ and $R_{\boldsymbol{\sigma},\Delta t}$ 
of the proposed method for Example~\ref{VariableceffPDE2}.}
\label{2dpdeAf41}

\endgroup
\end{table}
\end{example}

\begin{remark}
    In Example~\ref{VariableceffPDE2}, we achieve the optimal rate of convergence by strategically selecting $u_0 \in H_0^1(\Omega) \cap H^2(\Omega)$. To fully grasp this outcome, it is essential to explore the underlying theory.
    
\end{remark}
%
\section{Conclusion}\label{section7}
%

We introduce a non-uniform Alikhanov mixed finite element method for a class of time-fractional Allen-Cahn equations. The error estimates of these methods are derived for the considered problem using a modified discrete fractional Gr\"{o}nwall inequality. A global second-order error estimate with respect to the $L^{2}$-norm is obtained for the solution as well as flux with initial data $u_0\in H_0^1(\Omega) \cap H^{3+\epsilon}(\Omega)$ with $0<\epsilon\leq1$, up to a factor of $\log_e(N)$. These estimates remain valid as $\alpha\to 1^{-}$. Additionally, numerical experiments are conducted to validate our theoretical findings.
\section*{Acknowledgement}

Abhinav Jha acknowledges support from the Indian Institute of Technology Gandhinagar through grant No.~IP/IP/52016. 
Samir Karaa acknowledges support through the grant No.~IG/SCI/MATH/24/02. 
Aditi Tomar acknowledges support through the grant No.~IP/IP/52012.
\bibliographystyle{plain}
\bibliography{Time_Fractional}
\appendix

\section{Proof of  Lemma~\ref{DFGI} (Discrete fractional Gr\"{o}nwall inequality)}\label{appendix_sec_gronwall_b}
\begin{proof}
Apply the  definition of $D^{\alpha}_{t_{n-\nu}}$ in (\ref{dfgi1}) to get 
\begin{align*}
&\displaystyle{\sum_{k=1}^{j}}K^{j,k}_{1-\alpha} \left((v^k)^2-(v^{k-1})^2\right) \leq   \sum_{i=0}^{j} \lambda^j_{i}(v^{i})^2 + v^{j-\sigma}\xi^j + (\eta^j)^2 + (\zeta^j)^2, \quad 1\leq j \leq N.
\end{align*}  
After multiplying the above inequality by $P^{n,j}_{\alpha}$ and summing the index $j$ from $1$ to $n$, we obtain
\begin{align}\label{dfgieqn1}
\nonumber &\displaystyle{\sum_{j=1}^{n}P^{n,j}_{\alpha}\sum_{k=1}^{j}}K^{j,k}_{1-\alpha} \left((v^k)^2-(v^{k-1})^2\right) \leq\;  \sum_{j=1}^{n} P^{n,j}_{\alpha} \sum_{i=0}^{j} \lambda^j_{i}(v^{i})^2 \\
&\hspace{3cm}+ \sum_{j=1}^{n} P^{n,j}_{\alpha} v^{j-\nu}\xi^j + \sum_{j=1}^{n} P^{n,j}_{\alpha} (\eta^j)^2 + \sum_{j=1}^{n} P^{n,j}_{\alpha} (\zeta^j)^2,   
\end{align}
for $1\leq n \leq N$. An exchange of order of summation and Lemma~\ref{discretekernelproperties}\ref{itm:b} yields
\begin{align}
\label{dfgieqn2}\displaystyle{\sum_{j=1}^{n}P^{n,j}_{\alpha}\sum_{k=1}^{j}}K^{j,k}_{1-\alpha} \left((v^k)^2-(v^{k-1})^2\right) &= \sum_{k=1}^{n}\left(\sum_{j=k}^n P^{n,j}_{\alpha}K^{j,k}_{1-\alpha} \right)\left((v^k)^2-(v^{k-1})^2\right)\\
&= (v^n)^2 - (v^0)^2 , \quad 1\leq n\leq N.
\end{align}
Further, Lemma~\ref{discretekernelproperties}\ref{itm:c} implies
\begin{align}
\label{dfgieqn3}&\sum_{j=1}^{n} P^{n,j}_{\alpha}  \leq   \pi_A t_n^{\alpha}{\Gamma(1+\alpha)} \leq 2 \pi_A t_n^{\alpha} , \quad 1\leq n\leq N, 
\end{align}
where we have used $\Gamma(1+\alpha) \geq 2^{\alpha -1} \geq \frac{1}{2} ~ \forall ~\alpha \in [0,1]$. 
Thus, by using relations (\ref{dfgieqn2}) and (\ref{dfgieqn3}) in (\ref{dfgieqn1}), we obtain, for $1\leq n \leq N$,
\begin{align}\label{dfgieqn4}
& (v^n)^2  \leq (v^0)^2 + \sum_{j=1}^{n} P^{n,j}_{\alpha} \sum_{i=0}^{j} \lambda^j_{i}(v^{i})^2 + \sum_{j=1}^{n} P^{n,j}_{\alpha} v^{j-\sigma} \xi^j + 2 \pi_A t_{n}^{\alpha} \max_{1\leq j \leq n}(\eta^j)^2 \nonumber\\
&\hspace{1.5cm}+ \sum_{j=1}^{n} P^{n,j}_{\alpha} (\zeta^j)^2.
\end{align}  
Let us now define a non-decreasing finite sequence $\{\Phi_n \}_{n=1}^{N}$:
\begin{align*}
\Phi_n :=  v^0 + \max_{1 \leq j \leq n}\sum_{i=1}^j P^{j,i}_{\alpha}\xi^i + \sqrt{2 \pi_A} t_n^{\alpha}\max_{1 \leq j \leq n} \eta^j + \max_{1 \leq j \leq n} \left(\sum_{i=1}^{j} P^{j,i}_{\alpha} (\zeta^i)^2 \right)^{\frac{1}{2}},
\end{align*}
for $1\leq n \leq N$, and 
\begin{align*}
 F_n :=  C_\delta E_{\alpha}(C_\delta \pi_A \Lambda_n t_{n}^{\alpha})  
\end{align*}
derive the required estimate (\ref{dfgi2}), that is,
\begin{align}\label{dfgieqn5}
    v^n \leq F_n\Phi_n, \quad \forall ~ 1 \leq n \leq N
\end{align}
using mathematical induction in the following two steps:

\textbf{Step-I ($1\leq n \leq n_\alpha$):} As 
$$
\sum_{m=0}^{n}(C_\delta \pi_A \Lambda)^m k_{1+m\alpha}(t_n)\leq \sum_{m=0}^{\infty}(C_\delta \pi_A\Lambda)^m k_{1+m\alpha}(t_n)=:E_\alpha(C_\delta \pi_A\Lambda t_n^\alpha),\;1\leq n\leq N,
$$
it is enough to prove that
\begin{align}\label{dfgieqncaseI1}
    v^n \leq C_\delta \left(\sum_{m=0}^{n}(C_\delta \pi_A \Lambda)^m k_{1+m\alpha}(t_n)\right)\;\Phi_n \quad \forall ~ 1 \leq n\leq  n_\alpha.
\end{align}
For $n=1$, if $v^1 < v^0$ or $v^1 < \left(2 \pi_A t_1^\alpha\right)^{1/2}\eta^1$ or $v^1 < \left(P_{\alpha}^{1,1}(\zeta^1)^2\right)^{1/2}$, then the result (\ref{dfgieqncaseI1}) holds trivially. Otherwise, $v^0 \leq v^1$, $\left(2 \pi_A t_1^\alpha\right)^{1/2}\eta^1 \leq v^1$, $\left(P_{\alpha}^{1,1}(\zeta^1)^2\right)^{1/2} \leq v^1$, $v^{1-\sigma} = \sigma v^{0} + (1-\sigma)v^1 \leq v^1$ and hence, the inequality (\ref{dfgieqn4}) implies 
\begin{align}
\nonumber (v^1)^2
&\leq \left(v^0 + P^{1,1}_{\alpha}\lambda_{0}^{1} v^0 + P^{1,1}_{\alpha}\lambda_{1}^{1} v^1 +  P^{1,1}_{\alpha} \xi^1  + \left(2 \pi_A t_1^{\alpha}\right)^{1/2}\;\eta^1 + \left(P_{\alpha}^{1,1}(\zeta^1)^2\right)^{1/2} \right)v^1\\
\label{FDGI_Step-I1}&\leq (1+P^{1,1}_{\alpha}\lambda_{0}^{1})\Phi_1 v^1 +  P^{1,1}_{\alpha}\lambda_{1}^{1} (v^1)^2.
\end{align}
From condition (\ref{timecondition}) and Lemma~\ref{discretekernelproperties}\ref{itm:a}, we have $P^{1,1}_{\alpha}\lambda_{1}^{1} \leq 1/\delta,\;\delta>1.$ Moreover, using condition (\ref{timecondition}) along with Lemma~\ref{discretekernelproperties}\ref{itm:c} for $m=1$, we arrive at  $1+ P^{1,1}_{\alpha}\lambda_{0}^{1} \leq \sum_{m=0}^{1}(C_\delta \pi_A \Lambda)^mk_{1+m\alpha}(t_1)$. Thus, we obtain (\ref{dfgieqncaseI1}) for $n=1$ from (\ref{FDGI_Step-I1}).

Now, assume that (\ref{dfgieqncaseI1}) holds for $1 \leq k \leq m-1$, where $2\leq m \leq n_\alpha$. Then, there exists an integer $m_0$, $1 \leq m_0 \leq m-1$, such that $v^{m_0}= \max_{1 \leq j \leq m-1} v^j$. If $v^m \leq v^{m_0}$, the induction hypothesis, along with the properties $\Phi_n \leq \Phi_{n+1}$ and $k_{1+m\alpha}(s) \leq k_{1+m\alpha}(t)$ for $s \leq t$, yields the result (\ref{dfgieqncaseI1}) for $n = m$. Let $v^m > v^{m_0}.$ If $v^m < v^0$, or $v^m < \sqrt{2 \pi_A t_m^{\alpha}}\max_{1 \leq j \leq m}\eta^j$, or $v^m < \max_{1 \leq j \leq m}\left(\sum_{i=1}^jP_{\alpha}^{j,i}(\zeta^i)^2\right)^{\frac{1}{2}}$, then the inequality (\ref{dfgieqn5}) trivially holds for $n=m$. Otherwise, since $v^0 \leq v^m$, $\sqrt{2 \pi_A t_m^{\alpha}}\max_{1 \leq j \leq m}\eta^j \leq v^m$, $\max_{1 \leq j \leq m}\left(\sum_{i=1}^jP_{\alpha}^{j,i}(\zeta^i)^2\right)^{\frac{1}{2}} \leq v^m$, and 
$$
v^{m-\sigma} = \sigma v^{m-1} + (1-\sigma)v^m \leq v^m
$$ applying (\ref{dfgieqn4}) and condition (\ref{timecondition}) shows
\begin{align*}
(v^m)^2 & \leq \left(\Phi_m+ \sum_{j=1}^{m-1} P^{m,j}_{\alpha} \sum_{i=0}^j \lambda^j_{i}v^i  + P^{m,m}_{\alpha} \sum_{i=0}^{m-1} \lambda^m_{i}v^i + P^{m,m}_{\alpha}\lambda^m_{m}v^m  \right)v^m\\
 & \leq \left(\Phi_m+ \sum_{j=1}^{m-1} P^{m,j}_{\alpha} \sum_{i=0}^j \lambda^j_{i}v^i  + P^{m,m}_{\alpha} \sum_{i=0}^{m-1} \lambda^m_{i}v^i  \right)v^m + \frac{1}{\delta}(v^m)^2.  
\end{align*}
Thus, we obtain
\begin{align*}
v^m & \leq C_\delta\left(\Phi_m+ \sum_{j=1}^{m-1} P^{m,j}_{\alpha} \sum_{i=0}^j \lambda^j_{i}v^i  + P^{m,m}_{\alpha} \sum_{i=0}^{m-1} \lambda^m_{i}v^i  \right), 
\end{align*}
and by applying the induction hypothesis, along with the properties $\Phi_n \leq \Phi_{n+1}$ and $k_{1+m\alpha}(s) \leq k_{1+m\alpha}(t)$ for $s \leq t$, we arrive at
\begin{eqnarray}
\nonumber v^m & \leq & C_\delta\Bigg{(}\Phi_m+ C_\delta\Lambda \sum_{j=1}^{m-1} P^{m,j}_{\alpha} \sum_{i=0}^j (C_\delta \pi_A \Lambda)^i k_{1+i\alpha}(t_j)\Phi_j \nonumber \\
&&+C_\delta\Lambda P^{m,m}_{\alpha} \sum_{i=0}^{m-1} (C_\delta \pi_A \Lambda)^i k_{1+i\alpha}(t_{m-1})\Phi_{m-1}  \Bigg{)}\\
\label{FDGI_Step-I3}&\leq & C_\delta\left(1+ C_\delta\Lambda \sum_{j=1}^{m} P^{m,j}_{\alpha} \sum_{i=0}^{m-1} (C_\delta \pi_A \Lambda)^i k_{1+i\alpha}(t_j) \right)\Phi_m.\nonumber
\end{eqnarray}
From Lemma~\ref{discretekernelproperties}\ref{itm:c}, it follows that 
$$
C_\delta\Lambda \sum_{j=1}^{m} P^{m,j}_{\alpha} \sum_{i=0}^{m-1} (C_\delta \pi_A \Lambda)^i k_{1+i\alpha}(t_j) \leq  \sum_{i=1}^{m} (C_\delta \pi_A \Lambda)^i k_{1+i\alpha}(t_m)
$$
for $1\leq m\leq n_\alpha$. Applying this estimate in (\ref{FDGI_Step-I3}), we conclude that (\ref{dfgieqncaseI1}) holds for $n=m$. By mathematical induction, (\ref{dfgieqncaseI1}) and, as a consequence, (\ref{dfgieqn5}) hold for all $n$ such that $1 \leq n \leq n_\alpha$.

\textbf{Step-II ($n_\alpha +1\leq n \leq N$):} For $n=n_\alpha +1$, if $v^{n_\alpha +1}<\max_{0\leq j \leq n_\alpha} v^j$, or $v^{n_\alpha +1} < \sqrt{2 \pi_A t_{n_\alpha +1}^{\alpha}}\max_{1 \leq j \leq n_\alpha +1}\eta^j$, or $v^{n_\alpha +1} < \max_{1 \leq j \leq n_\alpha +1}\left(\sum_{i=1}^j P_{\alpha}^{j,i}(\zeta^i)^2 \right)^{\frac{1}{2}}$, then by applying the non-decreasing property of $E_\alpha(\cdot)$ and $\Phi_n$, the result (\ref{dfgieqn5}) follows for $n=n_\alpha +1$. Otherwise, $\max_{0\leq j \leq n_\alpha} v^j \leq v^{n_\alpha +1}$,  $ \sqrt{2 \pi_A t_{n_\alpha +1}^{\alpha}}\max_{1 \leq j \leq n_\alpha +1}\eta^j \leq v^{n_\alpha +1} $, $ \max_{1 \leq j \leq n_\alpha +1}\left(\sum_{i=1}^j P_{\alpha}^{j,i}(\zeta^i)^2\right)^{\frac{1}{2}}\leq v^{n_\alpha +1}$, and $v^{n_\alpha + 1-\sigma} = \sigma v^{n_\alpha} + (1-\sigma)v^{n_\alpha +1} \leq v^{n_\alpha +1}$ and therefore, by applying (\ref{dfgieqn4}) and condition (\ref{timecondition}), we obtain
\begin{eqnarray}
\nonumber (v^{n_\alpha +1})^2 & \leq & \Bigg{(}\Phi_{n_\alpha +1}+ \sum_{j=1}^{n_\alpha } P^{n_\alpha +1,j}_{\alpha} \sum_{i=0}^j \lambda^j_{i}v^i  + P^{n_\alpha +1,n_\alpha +1}_{\alpha} \sum_{i=0}^{n_\alpha } \lambda^{n_\alpha +1}_{i}v^i \nonumber \\
&&+ P^{n_\alpha +1,n_\alpha +1}_{\alpha}\lambda^{n_\alpha +1}_{n_\alpha +1}v^{n_\alpha +1}  \Bigg{)}v^{n_\alpha +1}\\
\label{FDGI_Step-I4} & \leq &\left(\Phi_{n_\alpha +1}+ \sum_{j=1}^{n_\alpha } P^{n_\alpha +1,j}_{\alpha} \sum_{i=0}^j \lambda^j_{i}v^i  + P^{n_\alpha +1,n_\alpha +1}_{\alpha} \sum_{i=0}^{n_\alpha } \lambda^{n_\alpha +1}_{i}v^i  \right)v^{n_\alpha +1} \nonumber \\
&&+ \frac{1}{\delta}(v^{n_\alpha +1})^2.\nonumber
\end{eqnarray}
Apply (\ref{dfgieqn5}) to (\ref{FDGI_Step-I4}) for $1 \leq n \leq n_\alpha$ to find that
\begin{eqnarray}
\nonumber v^{n_\alpha +1} &  \leq & C_\delta\Bigg{(}\Phi_{n_\alpha +1}+ C_\delta\Lambda \sum_{j=1}^{n_\alpha } P^{n_\alpha +1,j}_{\alpha} E_\alpha\left(C_\delta \pi_A  \Lambda t_j^\alpha\right)\Phi_j  \nonumber \\
&&+C_\delta\Lambda P^{n_\alpha +1,n_\alpha +1}_{\alpha} E_\alpha\left(C_\delta \pi_A  \Lambda t_{n_\alpha}^\alpha\right)\Phi_{n_\alpha}  \Bigg{)}\\
\label{FDGI_Step-I5}& \leq& C_\delta\left(1+ C_\delta\Lambda_{n_\alpha +1} \sum_{j=1}^{n_\alpha } P^{n_\alpha +1,j}_{\alpha} E_\alpha\left(C_\delta \pi_A  \Lambda_{n_\alpha +1} t_j^\alpha\right)   \right)\Phi_{n_\alpha +1}.\nonumber
\end{eqnarray}
Thus, an appeal to Lemma~\ref{discretekernelproperties}\ref{itm:d} in (\ref{FDGI_Step-I5}) yields the result (\ref{dfgieqn5}) for $n=n_\alpha +1$.

 Now, assume that (\ref{dfgieqn5}) holds for $1 \leq k \leq m-1$, where $n_{\alpha} + 2\leq m \leq N$. Then, there exists an integer $m_0$, $0 \leq m_0 \leq m-1$, such that $v^{m_0}= \max_{0 \leq j \leq m-1} v^j$. If $v^{m}<v^{m_0}$, or $v^{m} < \sqrt{2 \pi_A  t_{m}^{\alpha}}\max_{1 \leq j \leq m}\eta^j$, or $v^{m} < \max_{1 \leq j \leq m}\left(\sum_{i=1}^j P_{\alpha}^{j,i}(\zeta^i)^2 \right)^{\frac{1}{2}}$, then by applying the non-decreasing property of $E_\alpha(\cdot)$ and $\Phi_n$, the result (\ref{dfgieqn5}) follows for $n=m$. Otherwise, $v^{m_0}\leq v^{m}$, $ \sqrt{2 \pi_A  t_{m}^{\alpha}}\max_{1 \leq j \leq m}\eta^j\leq v^{m}$, and 
 $$
 \max_{1 \leq j \leq m}\left(\sum_{i=1}^j P_{\alpha}^{j,i}(\zeta^i)^2\right)^{\frac{1}{2}}\leq v^{m}
 $$
 , and hence, by applying (\ref{dfgieqn4}), condition (\ref{timecondition}), and the induction hypothesis, one can obtain 
 \begin{align}
\label{FDGI_Step-I6} v^{m} & \leq C_\delta\left(1+ C_\delta\Lambda_{m} \sum_{j=1}^{m-1 } P^{m,j}_{\alpha} E_\alpha\left(C_\delta \pi_A \Lambda_{m} t_j^\alpha\right)   \right)\Phi_{m}.
\end{align}
An appeal to Lemma~\ref{discretekernelproperties}\ref{itm:d} in (\ref{FDGI_Step-I6}) yields that (\ref{dfgieqn5}) holds for $n=m$. Thus, the principle of mathematical induction confirms that the result (\ref{dfgieqn5}) holds for all $n$, $1 \leq n \leq N$. Finally, using Lemma~\ref{discretekernelproperties}\ref{itm:g}, we can estimate $\Lambda_n$ in terms of $\mu_n=\Delta t_n/\Delta t_{n-1}$ as follows:
\begin{align*}
    & \Lambda_n \leq \Lambda \left(1+ \max\limits_{n_\alpha +1\leq j \leq n} \frac{2-\alpha}{\alpha} \mu_n^\alpha\right).
\end{align*}
This concludes the rest of the proof.
\end{proof}

\end{document}

%% file: packages.tex
\usepackage{multirow}
\usepackage{dsfont}
\usepackage{float}
\usepackage{verbatim}
\usepackage{amssymb}
\usepackage{amsmath, amsthm}
\usepackage{hyperref}
\usepackage{cleveref}
\usepackage{graphicx}
\usepackage{subfig}

\usepackage{enumitem}

\newcommand{\vertiii}[1]{\left\vert\kern-0.25ex\left\vert\kern-0.25ex\left\vert #1
\right\vert\kern-0.25ex\right\vert\kern-0.25ex\right\vert}

\crefname{section}{Section}{Sections}
\crefname{subsection}{subsection}{subsections}
\crefname{theorem}{Theorem}{Theorems}
\crefname{lemma}{Lemma}{Lemmas}
\crefname{proposition}{Proposition}{Propositions}
\crefname{corollary}{Corollary}{Corollaries}
\crefname{definition}{Definition}{Definitions}
\crefname{example}{Example}{Examples}
\crefname{remark}{Remark}{Remarks}
\crefname{table}{Table}{Tables}

\Crefname{figure}{Figure}{Figures}
\Crefname{equation}{}{}

\newtheorem{theorem}{Theorem}[section]
\newtheorem{lemma}[theorem]{Lemma}

\theoremstyle{definition}

\newtheorem{example}[theorem]{Example}

\theoremstyle{remark}
\newtheorem{remark}[theorem]{Remark}

\makeatletter
\newcommand{\namedlabel}[2]{%
  \begingroup
  #2%
  \def\@currentlabel{#2}%
  \phantomsection\label{#1}%
  \endgroup
}
\makeatother

\newcommand{\se}{\setcounter{equation}{0}}

